\documentclass[english]{article}
\usepackage[T1]{fontenc}
\usepackage[latin9]{inputenc}
\usepackage{geometry}
\usepackage{enumitem}
\usepackage{amsmath}
\usepackage{amsthm}
\usepackage{amssymb}
\usepackage{enumitem}
\usepackage{mathtools}
\usepackage[textsize=scriptsize,colorinlistoftodos]{todonotes}

\makeatletter
\numberwithin{equation}{section}
\numberwithin{figure}{section}
\theoremstyle{plain}
\newtheorem{thm}{\protect\theoremname}
  \theoremstyle{plain}
  \newtheorem{lem}[thm]{\protect\lemmaname}
  \theoremstyle{plain}
  \newtheorem{prop}[thm]{\protect\propositionname}
  \theoremstyle{plain}
  \newtheorem{cor}[thm]{\protect\corollaryname}
  \theoremstyle{plain}
  \newtheorem{conj}[thm]{\protect\conjecturename}
  \theoremstyle{definition}
  \newtheorem{defn}[thm]{\protect\definitionname}
 \newlist{casenv}{enumerate}{4}
 \setlist[casenv]{leftmargin=*,align=left,widest={iiii}}
 \setlist[casenv,1]{label={{\itshape\ \casename} \arabic*.},ref=\arabic*}
 \setlist[casenv,2]{label={{\itshape\ \casename} \roman*.},ref=\roman*}
 \setlist[casenv,3]{label={{\itshape\ \casename\ \alph*.}},ref=\alph*}
 \setlist[casenv,4]{label={{\itshape\ \casename} \arabic*.},ref=\arabic*}

\makeatother

\usepackage{babel}
  \providecommand{\corollaryname}{Corollary}
  \providecommand{\definitionname}{Definition}
  \providecommand{\lemmaname}{Lemma}
  \providecommand{\propositionname}{Proposition}
  \providecommand{\conjecturename}{Conjecture}
  \providecommand{\casename}{Case}
\providecommand{\theoremname}{Theorem}

\begin{document}

\title{Online Ramsey Numbers and \\ the Subgraph Query Problem}

\author{David Conlon\thanks{Mathematical Institute, Oxford OX2 6GG,
United Kingdom. Email: {\tt david.conlon@maths.ox.ac.uk}. Research
supported by a Royal Society University Research Fellowship and by ERC Starting Grant 676632.}\and Jacob Fox\thanks{Department of Mathematics, Stanford University, Stanford, CA 94305, USA. Email: {\tt jacobfox@stanford.edu}. Research supported by a Packard Fellowship and by NSF Career Award DMS-1352121.}\and Andrey Grinshpun\and Xiaoyu He\thanks{Department of Mathematics, Stanford University, Stanford, CA 94305, USA. Email: {\tt alkjash@stanford.edu}.}}
\maketitle

\begin{abstract}
The $(m,n)$-online Ramsey game is a combinatorial game between two players, Builder and Painter. Starting from an infinite set of isolated vertices, Builder draws an edge on each turn and Painter immediately paints it red or blue. Builder's goal is to force Painter to create either a red $K_m$ or a blue $K_n$ using as few turns as possible. The online Ramsey number $\tilde{r}(m,n)$ is the minimum number of edges Builder needs to guarantee a win in the $(m,n)$-online Ramsey game. By analyzing the special case where Painter plays randomly, we obtain an exponential improvement
\[
\tilde{r}(n,n) \ge 2^{(2-\sqrt{2})n + O(1)}
\]
for the lower bound on the diagonal online Ramsey number, as well as a corresponding improvement
\[
\tilde{r}(m,n) \ge n^{(2-\sqrt{2})m + O(1)}
\]
for the off-diagonal case, where $m\ge 3$ is fixed and $n\rightarrow\infty$. Using a different randomized Painter strategy, we prove that $\tilde{r}(3,n)=\tilde{\Theta}(n^3)$, determining this function up to a polylogarithmic factor. We also improve the upper bound in the off-diagonal case for $m \geq 4$. 

In connection with the online Ramsey game with a random Painter, we study the problem of finding a copy of a target graph $H$ in a sufficiently large unknown Erd\H{o}s--R\'{e}nyi random graph $G(N,p)$ using as few queries as possible, where each query reveals whether or not a particular pair of vertices are adjacent. We call this problem the Subgraph Query Problem. We determine the order of the number of queries needed for complete graphs up to five vertices and prove general bounds for this problem. 

\end{abstract}

\section{Introduction}

The \emph{Ramsey number} $r(m,n)$ is the minimum integer $N$ such that every red/blue-coloring of the edges of the complete graph $K_N$ on $N$ vertices contains either a red $K_m$ or a blue $K_n$. Ramsey's theorem guarantees the existence of $r(m,n)$ and determining or estimating Ramsey numbers is a central problem in combinatorics. Classical results of Erd\H{o}s--Szekeres and Erd\H{o}s imply that $2^{n/2} \leq r(n,n) \leq 2^{2n}$ for $n \geq 2$. The only improvements to these bounds over the last seventy years have been to lower order terms (see~\cite{Conlon0, Sp}), with the best known lower bound coming from an application of the Lov\'asz local lemma~\cite{ErLo}.

Off-diagonal Ramsey numbers, where $m$ is fixed and $n$ tends to infinity, have also received considerable attention. In progress that has closely mirrored and often instigated advances on the probabilistic method, we now know that 
\[
r(3,n)=\Theta(n^2/\log n).
\]
The lower bound here is due to Kim \cite{Ki} and the upper bound to Ajtai, Koml\'os and Szemer\'edi \cite{AjKoSz}. Recently, Bohman and Keevash \cite{BoKe1} and, independently, Fiz Pontiveros, Griffiths and Morris \cite{FiGrMo} improved the constant in Kim's lower bound via careful analysis of the triangle-free process, determining $r(3,n)$ up to a factor of $4+o(1)$.  

More generally, for $m\ge 4$ fixed and $n$ growing, the best known lower bound is 
\[
r(m,n) = \Omega_m(n^{\frac{m+1}{2}}/(\log{n})^{\frac{m+1}{2}-\frac{1}{m-2}}),
\]
proved by Bohman and Keevash \cite{BoKe} using the $H$-free process, while the best upper bound in this setting is 
\[
r(m,n) = O_m(n^{m-1}/(\log n)^{m-2}),
\]
again due to Ajtai, Koml\'os and Szemer\'edi \cite{AjKoSz}. Here the subscripts denote the variable(s) that the implicit constant is allowed to depend on.

There are many interesting variants of the classical Ramsey problem. One such variant is the \emph{size Ramsey number} $\hat{r}(m, n)$, defined as the smallest $N$ for which there exists a graph $G$ with $N$ edges such that every red/blue-coloring of the edges of $G$ contains either a red $K_m$ or a blue $K_n$. It was shown by Chv\'atal (see Theorem 1 in the foundational paper of Erd\H os, Faudree, Rousseau and Schelp \cite{ErFaRoSc}) that $\hat{r}(m, n)$ is just the number of edges in the complete graph on $r(m,n)$ vertices, that is,
\[
\hat{r}(m, n) = \binom{r(m,n)}{2}.
\]

We will be concerned with a much-studied game-theoretic variant of the size Ramsey number, introduced independently by Beck~\cite{Beck93} and by Kurek and Ruci\'nski~\cite{KuRu}. The $(m,n)$-online Ramsey game is a game between two players, Builder and Painter, on an infinite set of initially isolated vertices. Each turn, Builder places an edge between two nonadjacent vertices and Painter immediately paints it either red or blue. The \emph{online Ramsey number} $\tilde{r}(m,n)$ is then the smallest number of turns $N$ that Builder needs to guarantee the existence of either a red $K_m$ or a blue $K_n$.

It is a simple exercise to show that $\tilde{r}(m,n)$ is related to the usual Ramsey number $r(m,n)$ by
\begin{equation} \label{eq:trivial}
\frac{1}{2} r(m,n) \le \tilde{r}(m,n) \le \binom{r(m,n)}{2}.
\end{equation}
In the diagonal case, the upper bound in (\ref{eq:trivial}) has been improved by Conlon \cite{Conlon}, who showed that for infinitely many $n$,
\[
\tilde{r}(n,n) \le 1.001^{-n} \binom{r(n,n)}{2}.
\]

The main result of this paper is a new lower bound for online Ramsey numbers.

\begin{thm}\label{thm:general-lower-bound} If, for some $m,n,N\ge1$, there exist $p\in(0,1)$, $c\le\frac{1}{2}m$,
and $d\le\frac{1}{2}n$ for which
\[
p^{\binom{m}{2}-c(c-1)}(2N)^{m-c}+(1-p)^{\binom{n}{2}-d(d-1)}(2N)^{n-d} \le \frac{1}{2},
\]
then $\tilde{r}(m,n)>N$.
\end{thm}

In particular, if $\tilde{r}(n) :=\tilde{r}(n,n)$ is the diagonal online Ramsey number, Theorem~\ref{thm:general-lower-bound} can be used to improve the classical bound $\tilde{r}(n) \ge 2^{n/2 - 1}$ by an exponential factor. Indeed, taking $p=\frac{1}{2}$ and $c = d \approx(1-\frac{1}{\sqrt{2}})n$ in
Theorem~\ref{thm:general-lower-bound}, we get the following immediate corollary.

\begin{cor} \label{cor:diagonalonline}
For the diagonal online Ramsey numbers $\tilde{r}(n)$,
\[
\tilde{r}(n)\ge 2^{(2-\sqrt{2})n-O(1)}.
\]
\end{cor}

As for the off-diagonal case, when $m$ is fixed and $n\rightarrow \infty$, Theorem~\ref{thm:general-lower-bound} can be also used to substantially improve the best-known lower bound. In this case, we take $c\approx(1-\frac{1}{\sqrt{2}})m$, $d=0$, and $p=C\frac{m\log n}{n}$
for a sufficiently large $C>0$ to obtain the following corollary.

\begin{cor} \label{cor:offdiagonalonline}
For fixed $m\ge3$ and $n$ sufficiently large in terms of $m$, 
\[
\tilde{r}(m,n)\ge n^{(2-\sqrt{2})m-O(1)}.
\]
\end{cor}

For general $m$, Corollary~\ref{cor:offdiagonalonline} gives the best known lower bounds for the off-diagonal online Ramsey number. However, it is possible to do better for $m=3$ by using a smarter Painter strategy which deliberately avoids building red triangles.

\begin{thm} \label{thm:alterations} For $n\rightarrow \infty$,
\[
\tilde{r}(3,n) = \Omega \left(\frac{n^3}{\log^2 n}\right).
\]
\end{thm}

Roughly speaking, Painter's strategy is to paint every edge blue initially, but to switch to painting randomly if both endpoints of a freshly built edge have high degree. Also, when presented with an edge that would complete a red triangle, Painter always paints it blue. The bound given in Theorem~\ref{thm:alterations} is $n$ times the bound on the usual Ramsey number that comes from applying the Lov\'asz Local Lemma~\cite{ErLo}. However, our argument is closer in spirit to an earlier proof of the same bound given by Erd\H{o}s~\cite{Er} using alterations. This method for lower bounding $r(3,n)$ was later generalized to all $r(m,n)$ by Krivelevich~\cite{Kr1} and we suspect that Theorem~\ref{thm:alterations} can be generalized to $\tilde{r}(m,n)$ in the same way.

In the other direction, we prove a new upper bound on the off-diagonal online Ramsey number.

\begin{thm} \label{thm:online-upper-bound}
For any fixed $m\ge3$,
\[
\tilde{r}(m,n)= O_m\left(\frac{n^{m}}{\left(\log n\right)^{\lfloor m/2\rfloor-1}}\right).
\]
\end{thm}

In particular, note that Theorems~\ref{thm:alterations} and~\ref{thm:online-upper-bound} determine the asymptotic growth rate of $\tilde{r}(3,n)$ up to a polylogarithmic factor, namely,
\[\Omega\left(\frac{n^3}{\log^2 n}\right) \leq \tilde{r}(3,n) \leq O\left(n^3\right).\]

Theorem~\ref{thm:online-upper-bound} has a similar flavor to the improvement on diagonal online Ramsey numbers made by the first author~\cite{Conlon} and work on the so-called vertex online Ramsey numbers due to Conlon, Fox and Sudakov~\cite{CoFoSu}. It is obtained by adapting the standard Erd\H{o}s--Szekeres proof of Ramsey's theorem to the online setting and applying a classical result of Ajtai, Koml\'os and Szemer\'edi~\cite{AjKoSz} bounding $r(m,n)$.

In order to prove Theorem~\ref{thm:general-lower-bound}, we specialize to the case where Painter plays randomly. This is sufficient because Builder, who we may assume has unlimited computational resources, will always respond in the best possible manner to Painter's moves. Therefore, if a random Painter can stop this perfect Builder from winning within a certain number of moves with positive probability, an explicit strategy exists by which Painter can delay the game up to this point. This motivates the following key definition.

\begin{defn}
For $m,n \ge 3$ and $p\in(0,1)$, define $\tilde{r}(m,n;p)$ to be the number of turns Builder needs to win the $(m,n)$-online Ramsey game with probability at least $\frac{1}{2}$ against a Painter who independently paints each edge red with probability $p$ and blue with probability $1-p$. The \emph{online random Ramsey number} $\tilde{r}_{\text{rand}}(m,n)$ is the maximum value of $\tilde{r}(m,n;p)$ over $p\in(0,1)$.
\end{defn}

We note that there is a rich literature on simplifying the study of various combinatorial games by specializing to the case where one or both players play randomly (see~\cite{Beck, HKSS, Kr2}). For example, a variant of the online Ramsey game with random Builder instead of random Painter was studied by Friedgut et al. \cite{FKRRT}. 

We make the following conjectures about the growth rate of $\tilde{r}_{\text{rand}}(m,n)$.

\begin{conj}\label{conj:2/3}
\begin{enumerate}[label=(\alph*)]
\item The diagonal online random Ramsey numbers satisfy
\[
\tilde{r}_{\text{\normalfont rand}}(n,n) = 2^{(1+o(1))\frac{2}{3}n}.
\]
\item The off-diagonal online random Ramsey numbers ($m\ge 3$ fixed and $n\rightarrow\infty$) satisfy
\[
\tilde{r}_{\text{\normalfont rand}}(m,n) = n^{(1+o(1))\frac{2}{3}m}. 
\]

\end{enumerate}
\end{conj}

These conjectures are motivated by a connection with another problem, which we now describe. 

Let $p\in(0, 1)$ be a fixed probability and suppose Builder plays the following one-player game, which we call the {\it Subgraph Query Game}, on the random graph $G(\mathbb{Z},p)$ with infinitely many vertices. The edges of the graph are initially hidden. At each step, Builder queries a single pair of vertices and is told whether the pair is an edge of the graph or not. Equivalently, the graph starts out empty and each edge is successfully built by Builder with probability $p$ (each edge may be queried at most once). In what follows, we use the terms ``query'' and ``build'' interchangeably. 

Builder's goal is to find a copy of a given graph $H$ in the ambient random graph as quickly as possible. We call this problem of minimizing the number of steps in the Subgraph Query Game the \emph{Subgraph Query Problem}. When $H = K_m$, this may be seen as a variant of the online random Ramsey game, but where Builder is only interested in finding a red copy of $K_m$.

A version of this problem was studied independently by Ferber, Krivelevich, Sudakov and Vieira \cite{FKSV,FKSV2}, although they were interested in querying for long paths and cycles in $G(n,p)$. For instance, they showed that if $p \ge \frac{\log n + \log \log n + \omega(1)}{n}$, then it is possible to find a Hamiltonian cycle with high probability in $G(n,p)$ after $(1+o(1))n$ positive answers. In contrast, we are mainly interested in the setting where $H$ is a fixed graph to be found in a much larger random graph. 

\begin{defn} If $p\in(0,1)$, define $f(H,p)$ to be the minimum (over all Builder strategies) number of turns Builder needs to  be able to build a copy of $H$ with probability at least $1/2$ in the Subgraph Query Game, if each edge is built successfully with probability $p$.
\end{defn}

It might appear equally reasonable to study the minimum number of turns in which one can build at least one copy of $H$ {\it in expectation}. However, for certain $H$, such as a clique $K_m$ together with many leaves off a single vertex, it is possible to describe a strategy which has a tiny probability of successfully constructing copies of $H$, but upon success immediately builds a large number of copies, attaining low success probability but high expectation. Such a strategy is undesirable for application to online random Ramsey numbers, so we use the first definition instead.

Conjecture~\ref{conj:2/3} is motivated by the following conjecture regarding $f(K_m, p)$. The upper bound in this conjecture is proved in Section~\ref{sec:upper-bound-strat}.

\begin{conj} \label{conj:query-2/3} For any $m \ge 4$, 
\begin{equation*}
f(K_m, p) = 2^{o(m)}p^{-\frac{2}{3}m + c_m},
\end{equation*}
where
\[
c_m = \begin{cases}
\frac{m}{2m-3} & m\equiv0\pmod3\\
\frac{2}{3} & m\equiv1\pmod3\\
\frac{2m+8}{6m-3} & m\equiv2\pmod3.
\end{cases}
\]
\end{conj}

The following result shows that the Subgraph Query Problem and the online random Ramsey game are closely related.

\begin{thm} \label{thm:query-online-connection}
For any $m,n\ge 3$ and $p\in(0,1)$,
\begin{equation*} \label{eq:query-online-connection}
\tilde{r}(m,n;p) \le \min\{f(K_m, p), f(K_n,1-p)\} \le 3\tilde{r}(m,n;p).
\end{equation*}
\end{thm}

Using Theorem~\ref{thm:query-online-connection}, we can show that Conjecture~\ref{conj:query-2/3} implies both cases of Conjecture~\ref{conj:2/3}.
We can also determine an approximately optimal value for the probability parameter $p$ in the online Ramsey game with random Painter.

\begin{thm}\label{thm:opt-p}
For $m\ge 3$ fixed and $n\rightarrow\infty$, there exists a $p = \Theta(m/n\log(n/m))$ for which 
\[
\tilde{r}_{\text{\normalfont rand}}(m,n) \le 3 \tilde{r}(m,n; p).
\]
\end{thm}

We say that a graph has a \emph{$k$-matching} if it contains $k$ vertex-disjoint edges.
Our main result on the Subgraph Query Problem shows that graphs with large matchings are hard to build in few steps. We write $V(H)$ and $E(H)$ for the vertices and edges of $H$ and let $v(H) = |V(H)|$ and $e(H)=|E(H)|$.

\begin{thm}\label{thm:f-bound}
If $H$ is a graph that contains a $k$-matching, then
\[
f(H,p) = \Omega_H( p^{-(e(H)-k(k-1))/(v(H)-k)}).
\]
\end{thm}

Together with the upper bound construction described in Section~\ref{sec:upper-bound-strat}, this is enough to settle the growth rate of $f(K_m, p)$ for $m\le 5$. In particular, it proves Conjecture~\ref{conj:query-2/3} for $m=4, 5$.

\begin{thm} \label{thm:nle5}
The asymptotic growth rates of $f(K_m, p)$ for $m=3,4,5$ are
\begin{align*}
f(K_{3},p) & = \Theta(p^{-3/2})\\
f(K_{4},p) & = \Theta(p^{-2})\\
f(K_{5},p) & = \Theta(p^{-8/3}).
\end{align*}
\end{thm}

Asymptotically, the optimal $k$ to pick in Theorem~\ref{thm:f-bound} for $G = K_m$ is $k = (1-1/\sqrt{2})m$. With this value, we get the following bound on $f(K_m, p)$ which corresponds to Corollaries~\ref{cor:diagonalonline} and~\ref{cor:offdiagonalonline} in the online Ramsey number setting.

\begin{cor} \label{cor:f-bound-asymp}
For all $m\ge 3$,
\[
f(K_m,p) = \Omega_m(p^{-(2-\sqrt{2})m+O(1)}).
\]
\end{cor}

In studying the function $f(H,p)$, we were naturally led to consider the following function. When $H$ is a graph with no isolated vertices, define $t(H,p,N)$ to be the maximum expected number of copies of $H$ that can be built in $N$ moves in the Subgraph Query Game with parameter $p$, the maximum taken over all possible Builder strategies. 

However, if $H$ has isolated vertices, the expected value is zero or infinite. Instead, if $H$ has exactly $k$ isolated vertices $v_1,\ldots,v_k$, we define 
\[
t(H,p,N) \coloneqq (2N)^{k}t(H\backslash \{v_1,\ldots,v_k\},p,N)
\]
to capture the fact that the game with $N$ turns involves at most $2N$ vertices and therefore might as well be played on $2N$ fixed vertices.

Studying the threshold value of $N$ for which $t(H,p,N) \ge 1$ leads to Theorem~\ref{thm:f-bound} above. Intuitively, we expect the best strategy for building a copy of $H$ to be the same as the one which expects to build a single copy of $H$ in as few turns as possible.

Another natural question about the function $t(H,p,N)$ is: if $N$ is very large, what is the maximum number of copies Builder can expect to build in the Subgraph Query Game? Here we show that for $N$ sufficiently large the strategy of taking $O(\sqrt{2N})$ vertices and building all pairs of edges between them is asymptotically optimal for maximising $t(K_m,p,N)$, even though it is decidedly suboptimal for trying to build a single copy of $K_m$.

\begin{thm}\label{thm:t-large-N}
For all $m\ge 2, p \in (0,1), \varepsilon>0$, there exists $C > 0$ such that if $N\ge Cp^{-(2m-1)}(\log (p^{-1}))^2$, then
\[
t(K_{m},p,N) = (1\pm \varepsilon)p^{\binom{m}{2}}(2N)^{\frac{m}{2}}.
\]
\end{thm}

The rest of the paper is organized as follows.
In Section~\ref{sec:generallowerbounds}, we motivate and prove Theorem~\ref{thm:general-lower-bound}, our lower bound on the online Ramsey number, via the method of conditional expectations. In Section~\ref{sec:alterations}, we prove the lower bound Theorem~\ref{thm:alterations} for $\tilde{r}(3,n)$ using a Painter strategy designed to avoid red triangles.
We prove the upper bound Theorem~\ref{thm:online-upper-bound} in Section~\ref{sec:off-diagonal-upper}. Then, in Section~\ref{sec:subgraph-query}, we study the Subgraph Query Problem for its own sake, proving the upper bound in Conjecture~\ref{conj:query-2/3} as well as Theorems~\ref{thm:query-online-connection},~\ref{thm:opt-p},~\ref{thm:f-bound},~\ref{thm:nle5} and~\ref{thm:t-large-N}. We include a handful of open problems raised by our research in the closing remarks.

Unless otherwise indicated, all logarithms are base $e$. For clarity of presentation, we omit floor and ceiling signs when they are not crucial. We also do not attempt to optimize constant factors in the proofs.  

\section{General lower bounds}\label{sec:generallowerbounds}
\subsection{Motivation}

In this section, we prove Theorem~\ref{thm:general-lower-bound}
via a weighting argument, motivated by the method of conditional expectations
and a result of Alon \cite{Alon} on the maximum number of copies of a
given graph $H$ in a graph with a fixed number of edges.

The first idea, the derandomization technique known as the method
of conditional expectations (see Alon and Spencer \cite{AlSp}), can be
used to give the following ``deterministic'' proof of the classical lower
bound on diagonal Ramsey numbers. We will show that
\[
\binom{r(n, n)}{n}2^{-\binom{n}{2}+1}\ge1.
\]

Suppose that for some $N$,
\begin{equation}
\binom{N}{n}2^{-\binom{n}{2}+1}<1.\label{eq:totalweight}
\end{equation}
Paint the edges of $K_{N}$ one at a time as follows. To each vertex subset 
$U$ of order $n$, assign a weight $w(U)$ which is the
probability that $U$ becomes a monochromatic clique if the edges which remain uncolored at that time are colored uniformly at randomly. That is, writing $e(U)$ for the number of edges already colored in $U$,
\[
w(U)=\begin{cases}
2^{-\binom{n}{2}+1} & e(U)=0\\
2^{-\binom{n}{2}+e(U)} & \text{$e(U)>0$ and all already colored edges in $U$ are the same color}\\
0 & \text{otherwise}.
\end{cases}
\]
At every step, the total weight $\sum_{U}w(U)$
is equal to the expected number of monochromatic cliques if the remaining edges are painted uniformly at random. It is therefore possible to 
paint each edge so as not to increase the total weight. Since the condition 
$\sum_{U}w(U)<1$ is initially guaranteed by (\ref{eq:totalweight}), we can 
maintain this condition throughout the course of the game, ending with a coloring where there is no  
monochromatic clique of order $n$.

We now wish to apply such a weighting argument to the online
Ramsey game. The key observation is that if $\tilde{r}(n, n)$
is close to $r(n, n)$, then, since the graph built by Builder has at least 
$r(n, n)$ vertices, it must be extremely
sparse. In particular, most of the weight should be concentrated on sets $U$ 
almost none of whose edges are ever built.

This is where the idea behind Alon's result \cite{Alon} comes in. For any fixed graph $H$, that paper solves the problem of determining the maximum possible number of copies of $H$ in a graph with a prescribed number of edges. Roughly speaking, Alon showed that the maximum number of copies of $H$ can be controlled by the size of the maximum matching in $H$. We show that this heuristic also applies to the online Ramsey game, though it will be more convenient for our calculations to work with minimum vertex covers instead of maximum matchings. 

To make this idea work, instead of controlling the total weight function $\sum_{U}w(U)$, we restrict the sum to
subsets $U$ with a large minimum vertex cover, which
are comparatively few in number. Even if the total weight $\sum_{U}w(U)$
becomes large, the amount of weight supported on sets $U$ with a
large vertex cover is much smaller, and this is the only weight
that stands a chance to make it to the finish line and complete a 
monochromatic clique.

\subsection{The proof}

Using the weighting argument described informally above, we now prove a lower bound on the value of $\tilde{r}(m,n;p)$, where Painter plays randomly, independently coloring each edge red with probability $p$ and blue with probability $1-p$.

\begin{thm}\label{thm:random-lower-bound}
If, for some $m,n,N\ge1$ and $p\in(0,1)$, there exist $c\le\frac{1}{2}m$
and $d\le\frac{1}{2}n$ for which
\[
p^{\binom{m}{2}-c(c-1)}(2N)^{m-c}+(1-p)^{\binom{n}{2}-d(d-1)}(2N)^{n-d} \le \frac{1}{2},
\]
then $\tilde{r}(m,n;p)>N$.
\end{thm}

We would like to show that regardless of Builder's strategy, the online random Ramsey game lasts for more than $N$ steps with probability at least $1/2$.

Suppose the game ends in at most $N$ turns and, without loss of generality, is played on $2N$ vertices. Let $G_{t}$, for $0\le t\le N$, be the state
of the graph after $t$ turns. Assign to each subset $U\subset V(G)$
an evolving weight function
\begin{align*}
w(U,t) & = \begin{cases}
p^{\binom{|U|}{2}-e(G_{t}[U])} & G_{t}[U]\text{ is monochromatic red}\\
0 & \text{otherwise.}
\end{cases}
\end{align*}

The value of $w(U,t)$ is the probability that $U$ becomes a red
clique if the remaining edges are built.

We say that $C \subset V(G)$ is a {\it vertex cover} of $G$ if every edge is incident to some vertex $v \in C$. If $U\subset V(G)$, let $c(U,t)$ be the size of the minimum vertex cover of $G_{t}[U]$. Note that $c(U,t)$ is a nondecreasing function of $t$. For each pair $(k, c)$ with $k \geq 2c$, we will be interested in the total weight supported on sets of order $k$ with $c(U,t) \geq c$,
\begin{align*}
w_{k,c}(t) & = \sum_{|U|=k,c(U,t)\ge c}w(U,t).
\end{align*}

Since $w(U,N)$ is nonnegative and $w(U,N)=1$ if and only if $U$ is a red clique, we see that for all $c \le m/2$, $w_{m,c}(N)$ is an upper bound for
the number of red copies of $K_m$ built after $N$ turns. We
would like to upper bound the expected value of $w_{m,c}(N)$. 

\begin{lem}\label{lem:weight}
With $w_{m,c}(t)$ as above, regardless of Builder's strategy,
\[
\mathbb{E}w_{m,c}(N)\le p^{\binom{m}{2}-c(c-1)}(2N)^{m-c}.
\]
\end{lem}

\begin{proof}
Each $U$ with the property $c(U,N)\ge c$ first achieves this property
at a time $t_{c}(U)$. We say that $U$ is $c$-critical at this time.
Write 
\[
w_{k,c}^{*}(t)=\sum_{|U|=k,t_{c}(U)=t}w(U,t)
\]
to be the contribution of the $c$-critical sets $U$ to $w_{k,c}(t)$.
Crucially, if we focus on the family of $U$ for which $t_{c}(U)=t$,
their expected total weight will remain $w_{k,c}^{*}(t)$ indefinitely.
Thus,
\[
\mathbb{E}w_{k,c}(N)=\sum_{t\le N}\mathbb{E}w_{k,c}^{*}(t).
\]

Now, a set $U$ which is $c$-critical at time $t$ must be
the vertex-disjoint union of the edge $e_{t}$ that Builder builds
at time $t$ and a set $U'$ of size $k-2$ with a vertex cover of order $c-1$. Also, because $U$ has a vertex cover of order $c-1$ before adding
this edge $e_{t}$, the edges incident to $e_{t}$ must also be incident
to one of the $c-1$ vertices in the vertex cover of $U'$, so $e_{t}$
is incident to a total of at most $2c-2$ edges in $U$. It follows
that after turn $t=t_{c}(U),$
\[
w(U,t)\le p^{2k-2c-2}w(U',t),
\]
where in particular if $U'$ is already not monochromatic then neither
is $U$. The exponent comes from the fact that among the total $2(k-2)$
edges between $e_{t}$ and $U'$ at least $2(k-2)-2(c-1)=2k-2c-2$
are thus far unbuilt and still contribute factors of $p$ to the weight of
$w(U,t)$. Thus, since each $U'$ completes at most one set $U$ which is $c$-critical at time $t$,
\[w_{k,c}^{*}(t) \le p^{2k-2c-2}w_{k-2,c-1}(t).\]

Further, note that there can only be $c$-critical sets at time $t$ if $e_t$ is colored red, which occurs with probability $p$. Otherwise, $w_{k,c}^{*}(t) = 0$. Taking expectations and using the fact that $\mathbb{E}w_{k,m}(t)$
is nondecreasing in $t$ gives
\begin{align*}
\mathbb{E}w_{k,c}^{*}(t) & \le p \cdot \mathbb{E}[p^{2k-2c-2}w_{k-2,c-1}(t)] \\
 & \le p^{2k-2c-1}\mathbb{E}w_{k-2,c-1}(N).
\end{align*}

Summing over all $t$,
\[\mathbb{E}w_{k,c}(N) \le N\cdot p^{2k-2c-1}\mathbb{E}w_{k-2,c-1}(N).\]

Iterating this last inequality, we conclude that
\begin{align*}
\mathbb{E}w_{m,c}(N) & \le N^{c}\cdot p^{2mc-3c^{2}}\mathbb{E}w_{m-2c,0}(N)\\
 & \le N^{c}\cdot p^{2mc-3c^{2}}\cdot (2N)^{m-2c}p^{\binom{m-2c}{2}}\\
 & \le p^{\binom{m}{2}-c(c-1)}(2N)^{m-c},
\end{align*}
as desired.
\end{proof}

The same analysis with the blue weight function 
\begin{align*}
w'(U,t) & = \begin{cases}
(1-p)^{\binom{|U|}{2}-e(G_{t}[U])} & G_{t}[U]\text{ is monochromatic blue}\\
0 & \text{otherwise}
\end{cases}
\end{align*}
leads to the conclusion that $\mathbb{E}w'_{n,d}(N)\le (1-p)^{\binom{n}{2}-d(d-1)}(2N)^{n-d}$ for all $n \geq 2d$. The assumption of Theorem~\ref{thm:random-lower-bound} then implies that the expected number of red $K_m$ plus the expected number of blue $K_n$ is at most $1/2$. This implies that the probability of containing either is at most $1/2$, completing the proof of Theorem~\ref{thm:random-lower-bound}. Theorem~\ref{thm:general-lower-bound} follows as an immediate corollary. 

\section{Lower bound via alterations}
\label{sec:alterations}

In this section, we improve the lower bound for the off-diagonal online Ramsey numbers $\tilde{r}(3,n)$ using a different Painter strategy. Our proof extends an alteration argument of Erd\H os \cite{Er} which shows that
\[
r(3,n) \ge \frac{cn^2}{\log^2 n},
\]
for some constant $c>0$. The main idea of Erd\H os' proof was to show that in a random graph $G(r,p)$ with $p\approx r^{-1/2}$, only a small fraction of the edges need to be removed to destroy all triangles. Moreover, with high probability, removing these edges doesn't significantly affect the graph's independence number.

Our proof involves a randomized strategy which pays particular attention to avoiding red triangles. Instead of painting entirely randomly, Painter's strategy is modified in two ways to avoid creating red triangles. First, if an edge is built incident to a vertex of degree less than $(n-1)/4$, Painter always paints it blue. Second, if painting an edge red would create a red triangle, Painter again always paints it blue. In all other cases, Painter paints edges red with probability $p$ and blue with probability $1-p$.

In order to show that this Painter strategy works, we first prove a structural result about Erd\H{o}s--R\'enyi random graphs. Roughly speaking, this lemma implies that if an edge is removed from each triangle in $G(r,p)$, the remaining graph still has small independence number.

\begin{lem} \label{lem:rand-struct}
Suppose $n$ is sufficiently large, $p = 20 \log n / n$, $r = 10^{-6} n^2 /(\log n)^2$
and $G \sim G(r,p)$ is an Erd\H{o}s--R\'{e}nyi random graph. Then, with high probability, there does not exist a set $S\subset V(G)$ of order $n$ such that more than $\frac{n^2}{10}$ pairs of vertices in $S$ have a common neighbor outside $S$.
\begin{proof}
Let $E_1$ be the event that the maximum degree of $G$ is at most $2rp$. For a given vertex subset $S$ of order $n$, let $E_1(S)$ be the event that every vertex outside $S$ has at most $2rp$ neighbors in $S$. Thus, $E_1$ implies $E_1(S)$ for all $S$.

For a set $S$ of size $n$, let $E_2(S)$ be the event that at most $\frac{n^2}{10}$ pairs of vertices in $S$ have a common neighbor outside $S$ and let $E_2$ be the event that $E_2(S)$ holds for all $S$. We will show $E_1 \wedge E_2$ occurs w.h.p. which in turn implies that $E_2$ itself occurs w.h.p.

The distribution of $\deg(v)$ for a single vertex $v\in G$ is the binomial distribution $B(r-1,p)$. Using the Chernoff bound (see, e.g., Appendix A in \cite{AlSp}), we find that
\[
\Pr[\deg(v) > 2rp] < \left(\frac{e}{4}\right)^{rp} < \exp\left(-\frac{n}{5\cdot 10^5\log n}\right).
\]

Taking the union over all vertices of $G$,  it follows that
\[
\Pr[\overline{E_1}] < r\exp\left(-\frac{n}{5\cdot 10^5\log n}\right),
\]
so $E_1$ occurs w.h.p.

Fix a set $S$ of $n$ vertices. For $v \in V(G) \backslash S$, define $\deg_S(v)$ to be the number of neighbors of $v$ in $S$. Since $E_1$ implies $E_1(S)$, we have
\[
\Pr[E_1 \wedge \overline{E_2(S)}] \le \Pr[E_1(S) \wedge \overline{E_2(S)}].
\]
We will show that this last probability is so small that we may union bound over all $S$.

For $E_1(S)$ to occur, the possible values of $\deg_S(v)$ range through $[0, 2rp]$. We will cut off the bottom of this range and divide the rest into dyadic intervals. Let $D_0 = -1, D_1 = 4enp, D_2 = 8enp, D_3 = 16enp, \ldots, D_k = 2rp$ so that $D_i = 2D_{i-1}$ for each $2\le i \le k-1$ and $D_k \le 2D_{k-1}$. The number of intervals $k$ satisfies $k\le \log_2(r/n) \le 2\log n$. 

Define $d_i$ to be the number of $v \in V(G) \backslash S$ satisfying $D_{i-1} < \deg_{S}(v) \le D_i$. For $\overline{E_2 (S)}$ to occur, it must be the case that
\[
\sum_{v\not\in S} \binom{\deg_S(v)}{2} \ge \frac{n^2}{10},
\]
as the left hand side counts each pair in $S$ with a common neighbor outside $S$ at least once. In particular,
\begin{equation}\label{eq:degree-prof}
\sum_{i=1}^{k} d_i \binom{D_i}{2} \ge \frac{n^2}{10}.
\end{equation}

Notice that since $D_1 = 4enp = 80e\log n$ and $d_1 \le r$,
\[
d_1 \binom{D_1}{2} \le r \cdot D_1 ^2 = \frac{64e^2}{10^4} n^2 < \frac{n^2}{20},
\]
so at least half the contribution of (\ref{eq:degree-prof}) must come from $i\ge 2$. Thus,
\begin{equation}\label{eq:degree-prof-2}
\sum_{i=2}^{k} d_i \binom{D_i}{2} \ge \frac{n^2}{20}.
\end{equation}

We would like to bound the probability that $E_1(S)$ and (\ref{eq:degree-prof-2}) occur simultaneously. Let $T$ be the family of all sequences $(d_i)_{i=1}^k$ which sum to $r-n$ and satisfy (\ref{eq:degree-prof-2}). Given the choice of $(d_i)_{i=1}^k$, the number of ways to assign vertices to dyadic intervals $(D_{i-1},D_i]$ is at most $\binom{r-n}{d_1,d_2,\ldots, d_k}$. 

If $i\ge 2$ and a vertex $v$ is assigned to $(D_{i-1},D_i]$, the probability that $\deg_S(v)$ lies in that interval is at most
\[
\Pr[\deg_S(v) > D_{i-1}]\le \binom{n}{D_{i-1}} p^{D_{i-1}} \le \left(\frac{enp}{D_{i-1}}\right)^{D_{i-1}}.
\]
If $i = 1$, then we simply use the trivial bound $\Pr[\deg_S(v) \in (D_0, D_1]] \le 1$. Thus,
\begin{align*}
\Pr[E_1(S) \wedge \overline{E_2(S)}] &\le \sum_{(d_i)\in T} \binom{r-n}{d_1,d_2,\ldots, d_k} \prod_{i=1}^{k} \Pr[\deg_S(v) \in (D_{i-1}, D_i]] \\
& \le \sum_{(d_i)\in T} \binom{r-n}{d_1,d_2,\ldots, d_k} \prod_{i=2}^{k} \left( \left(\frac{enp}{D_{i-1}}\right)^{D_{i-1}}\right)^{d_i} \\
& \le \sum_{(d_i)\in T} \prod_{i=2}^{k} \left(r \cdot \left(\frac{enp}{D_{i-1}}\right)^{D_{i-1}}\right)^{d_i},
\end{align*}
where we used $\binom{r-n}{d_1, d_2,\ldots, d_k} < r^{d_2 + \cdots + d_k}$. Next, the number of compositions of $r-n$ into $k$ parts is at most $r^k$, so $|T| \le r^k$ and we have
\begin{align}
\Pr[E_1(S) \wedge \overline{E_2(S)}] & \le r^k \max_{(d_i)\in T} \prod_{i=2}^{k} \left(r \cdot \left(\frac{enp}{D_{i-1}}\right)^{D_{i-1}}\right)^{d_i} \nonumber \\
& \le r^k \max_{(d_i)\in T} \exp\left(\sum_{i=2}^k d_i \log A_i\right), \label{eq:deg-opt}
\end{align}
where
\[
A_i = r \cdot \left(\frac{enp}{D_{i-1}}\right)^{D_{i-1}}.
\]

It remains to maximize the exponent in (\ref{eq:deg-opt}) subject to (\ref{eq:degree-prof-2}). Consider the function
\[
f(D) = \frac{1}{D^2}\log \left(r \cdot \left(\frac{enp}{D}\right)^{D}\right) = \frac{\log r}{D^2} + \frac{\log(enp)}{D} - \frac{\log D}{D}.
\]
Notice that $D_1 = 4enp = 80e\log n$ so that for $D\ge D_1$,
\[
r \cdot \left(\frac{enp}{D}\right)^{D} \le 
r \cdot \left(\frac{enp}{D_1}\right)^{D_1}\le r\cdot 2^{-80e\log n} < 1.
\]
Thus, $f(D)$ takes negative values on $[D_1, D_k]$. Its derivative is
\[
f'(D) = - \frac{2\log r}{D^3} - \frac{\log{enp}}{D^2} + \frac{\log D}{D^2} - \frac{1}{D^2} = \frac{D(\log D - \log(e^2np)) - 2\log r}{D^3}.
\]
Since $r \le n^2$, we find that whenever $D \ge D_1 = 4enp = 80e\log n$,
\[
f'(D) \ge \frac{D\log(4/e) - 2\log r}{D^3} \ge \frac{80e\log(4/e)\cdot \log n - 4\log n}{D^3} > 0,
\]
and so $f(D)$ is monotonically increasing on $[D_1, D_k]$ and attains its maximum value at $D_k = 2rp$. With $2rp = 4\cdot 10^{-5} n/\log n$ and $n$ sufficiently large, observe that
\[
\left(\frac{enp}{2rp}\right)^{2rp} = \left(\frac{10^6e(\log n)^2}{2n}\right)^{4\cdot 10^{-5} n/\log n} \le \exp(-2\cdot 10^{-5} n),
\]
so that this maximum value is
\[
f(2rp) \le \frac{10^{10}(\log n)^2}{16 n^2} \cdot \log (n^2 \cdot \exp(-2\cdot 10^{-5} n)) \le - \frac{10^5(\log n)^2}{16 n}.
\] 

In particular, because $\binom{D}{2} \ge D^2/3$ for $D \ge 3$ and $f(D)$ is always negative,
\begin{align*}
\sum_{i=2}^k d_i \log A_i & = \sum_{i=2}^k d_i \binom{D_i}{2}\cdot \frac{\log A_i}{\binom{D_i}{2}} \\
& \le 3 \sum_{i=2}^k d_i \binom{D_i}{2}\cdot f(D_i) \\
& \le 3 f(D_k) \sum_{i=2}^k d_i \binom{D_i}{2} \\
& \le 3 f(2rp) \cdot  \frac{n^2}{20} \\
& \le - n (\log n)^2
\end{align*}
for any $(d_i)\in T$.

Returning to (\ref{eq:deg-opt}), it follows that
\[
\Pr[E_1(S) \wedge \overline{E_2(S)}] \le r^k \max_{(d_i)\in T} \exp\left(\sum_{i=2}^k d_i \log A_i\right) \le r^k \exp(-n (\log n)^2).
\]
There are at most $\binom{r}{n} \le e^{2n\log n}$ subsets $S$ of size $n$ to consider and $r^k \le r^n \le e^{2n \log n}$ as well, so
\begin{align*}
\Pr[\overline{E_1 \wedge E_2}] & = \Pr[\overline{E_1} \vee  \bigvee_S E_2(S)] \\
& \le \Pr[\overline{E_1}] + \sum_S \Pr[E_1 \wedge \overline{E_2(S)}] \\
& \le \Pr[\overline{E_1}] + \sum_S \Pr[E_1(S) \wedge \overline{E_2(S)}] \\
& \le \Pr[\overline{E_1}] + \exp(4n\log n) \cdot \exp\left(-n(\log n)^2\right). \\
\end{align*}
Both summands on the right vanish rapidly, so $E_2$ holds w.h.p., as desired.
\end{proof}
\end{lem}

With this lemma in hand, we are now ready to prove Theorem~\ref{thm:alterations}.

\vspace{3mm}
\noindent
{\it Proof of Theorem~\ref{thm:alterations}.}
Let $p = 20 \log n / n$, $r = 10^{-6} n^2/(\log n)^2$ and $N=\frac{(n-1)r}{8}$.

We will give a randomized strategy for Painter such that, regardless of Builder's strategy, after $N$ edges are colored there is neither a red $K_3$ nor a blue $K_n$ w.h.p. Thus, there exists a strategy for Painter which makes the game last more than $N$ steps and the desired bound $\tilde{r}(3,n)>N$ follows. Note that proving the result with positive probability suffices, but our argument shows it w.h.p.~for no additional cost.

We now describe Painter's strategy. Initially, all vertices are considered inactive; a vertex is activated when its degree reaches at least $(n-1)/4$. The active vertices are labeled with the natural numbers in $[r]$ when they reach degree at least $(n-1)/4$, using an arbitrary underlying order on the vertices to break ties. Since $N = (n-1)r/8$, there will never be more than $r$ active vertices.

When Builder builds an edge $(u,v)$, this edge is considered inactive if either $u$ or $v$ is inactive immediately after $(u,v)$ is built and active otherwise. The status of an edge remains fixed once it is built, so that inactive edges remain inactive even if both of its incident vertices are active at a later turn. Painter automatically colors inactive edges blue.

If Builder builds an active edge $(u,v)$, Painter first checks if $u$ and $v$ have a common neighbor $w$ such that $(u,w)$ and $(v,w)$ are both red. For brevity's sake, we call such a vertex $w$ a {\it red common neighbor} of $u$ and $v$. If so, Painter paints $(u,v)$ blue so as to not build a red triangle and we call such an edge {\it altered}. Otherwise, Painter paints it red with probability $p$ and blue with probability $1-p$. Following this strategy, Painter guarantees that no red triangles are built. It suffices to show that w.h.p.~no blue $K_n$ is built either.

Here is an equivalent formulation of Painter's strategy. At the start of the game, Painter samples an Erd\H{o}s--R\'enyi graph $G = G([r],p)$ on the labels which he keeps hidden from Builder. Inactive edges are painted blue. When an active edge between vertices labelled $i$ and $j$ is built, it is painted red if and only if $i\sim j$ in $G$ and these two vertices currently have no red common neighbor.

Now, we apply Lemma~\ref{lem:rand-struct} to the graph $G$. Letting $E_2(S)$ be the event that an $n$-set $S$ has at most $n^2/10$ pairs with outside common neighbors and $E_2 = \bigwedge_S E_2(S)$, we see that $\Pr[\overline{E_2}] \rightarrow 0$ as $n\rightarrow \infty$.

For a set $S\subseteq [r]$ of labels, write $T(S)$ for the set of active vertices with labels in $S$.  We seek to bound the probability of the event $B(T(S))$ that $T(S)$ is a blue $n$-clique at the end of the game. Because any blue $n$-clique would have all of its vertices active (as each vertex of the $n$-clique would have degree at least $n-1\ge (n-1)/4$), if none of the events $B(T(S))$ occurs, then no blue $K_n$ is ever built. Once we show that the probability of a single $B(T(S))$ is sufficiently small, we will apply the union bound over all $S$ to show that w.h.p.~no blue $K_n$ is built.

First, note that if any edge $(u,v)$ in $T(S)$ is altered (and hence blue), we may assume that their red common neighbors lie outside $T(S)$. Otherwise, there must be two red edges inside $T(S)$ already and $T(S)$ can never become a blue $n$-clique.

With this in mind, conditioning on the event $E_2(S)$, at most $n^2/10$ altered blue edges are built in $T(S)$. Within $T(S)$ there can be at most $n^2/4$ inactive edges. Assuming $B(T(S))$ occurs, there are at least
\[
\binom{n}{2} - \frac{n^2}{4} - \frac{n^2}{10} \ge \frac{n^2}{8}
\]
edges between vertices of $T(S)$ that are both active and unaltered. For $B(T(S))$ to occur, each of these active and unaltered edges would have to be colored blue on its turn. On the other hand, each of these edges has a chance $p$ of being colored red on that turn.

Thus, we find that
\[
\Pr[B(T(S))|E_2(S)] \le (1-p)^{\frac{n^2}{8}},
\]
with one factor of $1-p$ for each unaltered active edge built in $T(S)$. Thus,
\[
\Pr[\bigvee_S B(T(S))] \le \Pr[E_2 \wedge \bigvee_S B(T(S))] + \Pr[\overline{E_2}].
\]
The second summand goes to zero, so it suffices to show the first does as well. We have
\begin{align*}
\Pr[E_2 \wedge \bigvee_S B(T(S))] & \le \sum_S \Pr[E_2 \wedge B(T(S))] \\
& \le \sum_S \Pr[E_2(S) \wedge B(T(S))] \\
& \le \sum_S \Pr[B(T(S))|E_2(S)] \\
& \le \binom{r}{n} (1-p)^{\frac{n^2}{8}}.
\end{align*}
Using $1-p \le e^{-p}$, the right-hand side is at most
\[
r^n e^{-pn^2/8} \le e^{n\log r -pn^2/8} = e^{-(\frac{1}{2} + o(1))n\log n},
\]
also tending to zero as $n\rightarrow \infty$. Thus, the probability that either $\overline{E_2}$ or some $B(T(S))$ occurs tends to zero. Therefore, with high probability no blue $K_n$ is built.
\hfill \qed

\section{Off-diagonal upper bounds} \label{sec:off-diagonal-upper}

In Section~\ref{sec:generallowerbounds}, we proved lower bounds of the form $\tilde{r}(m,n)\ge\Omega(n^{(2-\sqrt{2})m+o(m)})$ on the off-diagonal online Ramsey numbers through an analysis of the online random 
Ramsey number. It is easy to give an upper
bound of the form $\tilde{r}(m,n) \leq O(n^{2m-2})$ simply by applying
the Erd\H os--Szekeres bound for classical Ramsey numbers and the
trivial observation that $\tilde{r}(m,n)\le\binom{r(m,n)}{2}$.

However, the simple inductive proof of the Erd\H os--Szekeres bound suggests a 
Builder strategy that does considerably better. Namely, build many edges from one 
vertex until it has a large number of edges of one color, then proceed 
inductively in that neighborhood.
This strategy is particularly well suited to the online Ramsey game because
the number of edges built is only slightly more than linear in the
number of vertices used, allowing us to derive a bound of the form $\tilde{r}(m,n)\le O(n^{m})$. 

A slight variation on this argument allows us to bound the online Ramsey number in terms of the bounds for classical Ramsey numbers.

\begin{lem} \label{lem:classical-to-online-upper-bound}
Let $m \leq n$ be positive integers with $m$ fixed. Let $m_0=\lfloor m/2 \rfloor + 1$ and $n_0=\lfloor \sqrt{n} \rfloor$. Suppose $\mathcal{L}$ is a positive real such that for all $m_0 \le m' \le m$ and $n_0 \le n' \le n$,
\begin{align*}
r(m_0, n') & \le \frac{1}{\mathcal{L}}\binom{m_{0}+n'-2}{m_{0}-1}, \\
r(m',n_0) & \le \frac{1}{\mathcal{L}}\binom{m'+n_{0}-2}{m'-1}.
\end{align*}

Then 
\[
\tilde{r}(m,n)\le \frac{C_m n}{\mathcal{L}}\binom{m+n-2}{m-1}
\]
for a constant $C_m$ depending only on $m$.
\end{lem}

\begin{proof}
We describe a general Builder strategy for the online Ramsey game
with parameters $m$ and $n$ and some savings parameter $\mathcal{L}$. Let $f(m,n)=\frac{1}{\mathcal{L}}\binom{m+n-2}{m-1}$,
so we have $f(m-1,n)+f(m,n-1)=f(m,n)$ by Pascal's identity.

Begin by building $f(m,n) - 1$ edges out of a given initial vertex $v_{1}$.
If $f(m-1,n)$ of these edges are colored red, we proceed to
the red neighborhood of $v_1$; otherwise, we proceed to
the at least $f(m,n-1)$ vertices in the blue neighborhood of $v_1$. If at some step we reach a neighborhood with $f(m - i, n - j)$ vertices, we build $f(m - i, n -j) - 
1$ edges inside this neighborhood from one of the vertices, which we label 
$v_{i+j+1}$. If $f(m - i - 1, n - j)$ of these edges are colored red, we proceed 
to the red neighborhood of $v_{i+j+1}$; otherwise, we proceed to the at least 
$f(m-i,n-j-1)$ vertices in the blue neighborhood of $v_{i+j+1}$.
We stop once $m$ reaches $m_0$ or $n$ reaches $n_0$, ending up with either 
$f(m_{0},n')$ vertices for some $n_{0}\le n'\le n$ or $f(m',n_{0})$ vertices for some $m_{0}\le m'\le m$. Once we reach this stage, we build all edges in the remaining set.

Suppose now that we arrive at a set $S$ of order $f(m_0, n')$. By construction, there are $\ell = m + n - m_0 - n'$ vertices $v_1, \dots, v_\ell$ such that $m-m_0$ of the vertices $v_i$ are joined in red to every $v_j$ with $j > i$ and every $w \in S$. The remaining $n - n'$ vertices $v_i$ are joined in blue to every $v_j$ with $j > i$ and every $w \in S$. But since
\[
r(m_0, n') \le \frac{1}{\mathcal{L}}\binom{m_{0}+n'-2}{m_{0}-1} = f(m_0,n'),
\]
the complete graph on $S$ contains either a red $K_{m_0}$ or a blue $K_{n'}$, either of which can be completed to a red $K_m$ or a blue $K_n$ by using the appropriate subset of $v_1, \dots, v_\ell$. If we had instead arrived at a set of order $f(m',n_0)$, a similar analysis would have applied.

Note that the total number of edges built in the branching phase is at most
$(m+n)f(m,n)$, while the number built by filling in the final clique is at most $
\max(f(m_{0},n)^{2}, f(m,n_{0})^{2})$. Using the choice of $m_0$ and $n_0$, the total number of edges built is easily seen to be at most a constant in $m$ times the previous expression. 
\end{proof}

From here we derive Theorem~\ref{thm:online-upper-bound}.

\vspace{3mm}
\noindent
{\it Proof of Theorem~\ref{thm:online-upper-bound}.} 
We apply the bound
\[
r(m,n) = O_m(n^{m-1}/\log ^{m-2} n),
\]
due to Ajtai, Koml\'os and Szemer\'edi~\cite{AjKoSz}. In particular, suppose $m_0=\lfloor m/2 \rfloor + 1$, $n_0=\lfloor \sqrt{n} \rfloor$ and $m', n'$ satisfy $m_0 \le m' \le m$ and $n_0 \le n' \le n$. Then, for some constants $C, C' >0$ depending only on $m$, we have
\[
r(m_0, n') \le \frac{C}{\log ^{m_0-2} n'} (n')^{m_0-1} \le \frac{C'}{\log^{\lfloor m/2 \rfloor - 1} n}\binom{m_{0}+n'-2}{m_{0}-1}
\]
and
\[
r(m', n_0) \le \frac{C}{\log ^{m'-2} n_0} n_0^{m'-1} \le \frac{C'}{\log^{\lfloor m/2 \rfloor - 1} n}\binom{m'+n_{0}-2}{m'-1},
\]
verifying the conditions of Lemma~\ref{lem:classical-to-online-upper-bound} with $\mathcal{L} = \Omega_m(\log^{\lfloor m/2 \rfloor - 1} n)$. It follows by that lemma that there exists another constant $C''>0$ depending only on $m$ for which
\[
\tilde{r}(m,n) \le \frac{C'' n}{\log^{\lfloor m/2 \rfloor - 1} n} \binom{m+n-2}{m-1}.
\]

Fixing $m\ge 3$ and taking $n\rightarrow\infty$, this implies
\[
\tilde{r}(m,n)= O_m\left(\frac{n^{m}}{\left(\log n\right)^{\lfloor m/2\rfloor-1}}\right),
\]
as desired.
\hfill \qedsymbol

\vspace{3mm}

We remark that while the statement and proof of Lemma~\ref{lem:classical-to-online-upper-bound} are designed for the case where $m$ is a constant, they can be easily modified to make them meaningful for all $m$ and $n$. 

\section{The Subgraph Query Problem} \label{sec:subgraph-query}

The vertex cover argument in Section~\ref{sec:generallowerbounds} was motivated by our study of the closely-related Subgraph Query Problem. Indeed, one can view this problem as an instance of the online Ramsey game with a random Painter where Builder single-mindedly tries to build a clique in one color, ignoring the other color entirely.

Let $p\in(0,1)$ be the probability that Builder successfully builds any given edge in the Subgraph Query Problem. We are primarily interested in the quantity $f(H,p)$, which we defined as the minimum $N$ for which there exists a Builder strategy which builds a copy of $H$ with probability at least $\frac{1}{2}$ in $N$ turns. Of secondary interest is the quantity $t(H,p,N)$, which we define as the maximum, over all Builder strategies, of the expected number of copies of $H$ that can be built in $N$ turns. It is easy to see that
\begin{equation*}
t(H,p,N) < \frac{1}{2} \implies f(H,p) > N.
\end{equation*}
Thus, upper bounds on $t(H,p,N)$ yield lower bounds on $f(H,p)$.

\subsection{Connection with online Ramsey numbers}\label{sec:query-online-connection}

We first check that the Subgraph Query Problem gets easier when edges are built with higher probability.

\begin{lem}\label{lem:f-monotone}
For any $m\ge 3$, $f(K_m,p)$ is a nonincreasing function of $p\in (0,1)$.
\end{lem}
\begin{proof}
Suppose $p < q$ and $f(K_m, p) = N$. This means that in the Subgraph Query Problem with parameter $p$, Builder has an $N$-move strategy $S$ to win with probability at least half. Strategy $S$ is defined by Builder's choice of edge to build at each step, given the data of which edges were successfully built in previous steps.

Builder's strategy for the Subgraph Query Problem with parameter $q$ is as follows. For each edge that Builder successfully builds, Builder then flips a biased coin that comes up heads $\frac{p}{q}$ of the time. If the coin comes up tails, Builder pretends the edge actually failed to build, and acts according to strategy $S$ with respect to only the edges for which the coin came up heads. 
Just looking at the edges which come up heads, Builder is exactly following strategy $S$, and so builds a $K_m$ with probability at least $1/2$ in $N$ steps.
\end{proof}

We now prove Theorem~\ref{thm:query-online-connection}, which connects the Subgraph Query Problem to the online Ramsey game. Recall the statement:
\[
\tilde{r}(m,n;p) \le \min\{f(K_m, p), f(K_n,1-p)\} \le 3\tilde{r}(m,n;p).
\]

\vspace{3mm}
\noindent
{\it Proof of Theorem \ref{thm:query-online-connection}.}
We first show the left side of the inequality. Let $N=\min\{f(K_m, p), f(K_n,1-p)\}$ and suppose that $f(K_m,p)$ is the smaller of the two. Then there exists an $N$-move Builder strategy which builds a $K_m$ with probability at least half. Now, let Builder play the online Ramsey game against a random Painter with the same probability parameter $p$. Builder's strategy will be to treat red edges as successfully built and blue edges as failed. In this way, Builder wins the online Ramsey game in $N$ moves with probability at least half, by constructing a red $K_m$. Similarly, if $f(K_n, 1-p)$ were smaller, Builder would instead treat blue edges as successfully built and red edges as failed. This would then guarantee the construction of a blue $K_n$ with probability at least half.

Now we show the right side of the inequality. Suppose $N=\tilde{r}(m,n;p)$, so in the online Ramsey game against random Painter with parameter $p$, there exists an $N$-move Builder strategy which builds a red $K_m$ or blue $K_n$ with probability at least half. In particular, this same strategy guarantees either a red $K_m$ with probability at least $\frac{1}{4}$ or a blue $K_n$ with probability at least $\frac{1}{4}$.

Suppose the first is true. Then Builder plays the Subgraph Query Game using this same strategy, treating red edges as successfully built and blue as failed. In $N$ moves, he has at least a $\frac{1}{4}$ chance of successfully building a $K_m$. Repeating this strategy three independent times on three different vertex sets, Builder uses $3N$ moves to build a $K_m$ with probability at least
\[
1-\Big(1-\frac{1}{4}\Big)^3 = \frac{37}{64} > \frac{1}{2},
\]
showing that $f(K_m,p) \le 3\tilde{r}(m,n;p)$ in this case. Similarly, if the second case occurs, $f(K_n,1-p) \le 3\tilde{r}(m,n;p)$. Either way, the smaller of $f(K_m,p)$ and $f(K_n, 1-p)$ is bounded above by $3\tilde{r}(m,n;p)$.
\hfill \qed

\vspace{3mm}

Now we show that Conjecture~\ref{conj:query-2/3} about the Subgraph Query Problem directly implies Conjecture~\ref{conj:2/3} about online random Ramsey numbers.

\vspace{3mm}
\noindent
{\it Proof that Conjecture~\ref{conj:query-2/3} implies Conjecture~\ref{conj:2/3}.}
Assume Conjecture~\ref{conj:query-2/3}, i.e., $f(K_m, p) = 2^{o(m)}p^{-\frac{2}{3} m + c_m}$ for all $m\ge 3$, $p\in(0,1)$.
By Theorem~\ref{thm:query-online-connection}, we have
\begin{equation}\label{eq:online-to-f}
\tilde{r}(m,n;p) = \Theta(\min\{f(K_m,p), f(K_n,1-p)\}).
\end{equation}

In the diagonal case of the online Ramsey game, (\ref{eq:online-to-f}) together with Lemma \ref{lem:f-monotone} implies that $p=\frac{1}{2}$ gives the online random Ramsey number to within a constant factor. Thus,
\[
\tilde{r}_{\text{\normalfont rand}}(n,n) = 2^{\frac{2}{3} n + o(n)}.
\]
This proves part (a).

In the off-diagonal case, a value of $p$ nearly optimizing the right hand side of (\ref{eq:online-to-f}) satisfies $p = \Theta(\frac{m}{n}\log{\frac{n}{m}})$ by Theorem~\ref{thm:opt-p}, which is proved in Subsection \ref{sec:recursive-graph-building}. Plugging in this value of $p$, we get
\[
\tilde{r}_{\text{\normalfont rand}}(m,n)=2^{o(m)}\Big(\Theta\left(\frac{m}{n}\log{\frac{n}{m}}\right)\Big)^{-\frac{2}{3}m+c_m},
\]
which implies case (b) of Conjecture~\ref{conj:2/3}.
\hfill \qed

\vspace{3mm}

\subsection{The Branch and Fill Strategy}  \label{sec:upper-bound-strat}

We now prove the upper bound in Conjecture~\ref{conj:query-2/3}.

We will say it is possible to build a graph $H$ in $O(T)$ turns, where $T=T(p)$ is a function of $p$, if for any $p\in(0,1)$ it is possible, in the Subgraph Query Game played with probability $p$, to build a copy of $H$ in $O(T)$ time with probability at least $\frac{1}{2}$. It is a simple fact about randomized algorithms that if one can achieve any constant success probability in $O(T)$ time then one can iterate the algorithm to succeed with probability $1-\varepsilon$ in $O(T \log{\varepsilon^{-1}})$ time.

We describe a Builder strategy to prove the upper bound in Conjecture~\ref{conj:query-2/3} and conjecture that this is essentially the optimal strategy for the Subgraph Query Problem for cliques.

\begin{lem} \label{lem:branch-and-fill}
Let $a\ge 1$, $b \ge 2$ and $n=a+b+1$ satisfy $2a+3-b\ge 0$. Then
\[
f(K_{n},p) = O_n(p^{-\frac{2a+b+1}{2}+\frac{\alpha}{b}}),
\]
where $\alpha = \min(1, \frac{b(2a+3-b)}{2(b-1)})$.
\end{lem}

\begin{proof}
To build a clique $K_{n}$ in $O(T)$ turns, where $T = p^{-\frac{2a+b+1}{2}+\frac{\alpha}{b}}$, we follow a strategy with three phases:

\begin{enumerate}

\item Build a clique $U$ on $a$ vertices. By induction, the number of
turns needed will be negligible.

\item Find $p^{a}T$ common neighbors of $U$ in $O_n(T)$
time with high probability. This is done by repeatedly picking a new
vertex $v$ and trying to build each of the edges between $v$ and the vertices in
$U$ until one fails. Let $W$ be the set of common neighbors found in
this way.

\item Among the vertices of $W$, pick a vertex $w_{1}$
and try to build all edges incident to $w_{1}$ within $W$. Let $W_{1}=N(w_{1})\cap W$
be the neighborhood determined. Try to build all $\binom{|W_{1}|}{2}$
edges within $W_{1}$. Remove $\{w_{1}\}\cup W_{1}$ from $W$ and
repeat a total of $p^{-\alpha}$ times, picking $w_{2},\ldots,w_{p^{-\alpha}}$,
finding their neighborhoods, and filling them in. Here, $\alpha\in[0,1]$ is a parameter which we have not yet specified.

\end{enumerate}
After the process is complete, if any one of the $W_{i}$ contains
a $b$-clique $W_{i}'$, then we are done, since $U\cup\{w_{i}\}\cup W_{i}'$ forms an $n$-clique. 

It remains
to determine the success probability and the number of steps taken
in the above process. By the standard Chernoff bounds, the sizes of all
the sets $W_{i}$ concentrate around their means with high
probability. Hence, with high probability,
\[|W_{i}| = (1+o(1))p^{a+1}(1-p)^{i-1}T.\]
A standard application of Janson's inequality (see Chapters 8 and 10 of~\cite{AlSp}) then implies
\[
\Pr[W_{i}\text{ contains a \ensuremath{b}-clique}] = \Omega_b\left(\min(p^{\binom{b}{2}} |W_i|^b, 1)\right) = \Omega_{b}\left(\min(p^{(a+1)b+\binom{b}{2}}(1-p)^{(i-1)b}T^{b}, 1)\right).
\]
If $i$ ranges up to $p^{-\alpha}$ and $\alpha\le1$, then the decay
factor $(1-p)^{(i-1)b}$ is $\Theta_{b}(1)$ and can be safely ignored.
Since the event that each $W_{i}$ contains a $b$-clique is independent
of all the others, we need only pick $p,T,\alpha$ for which the expression
$p^{-\alpha}p^{(a+1)b+\binom{b}{2}}T^{b}$ is a positive
constant. If this is the case, then with at least constant probability
our strategy constructs an $n$-clique.

We also need to know that the total number of turns taken is $O_n(T)$.
This is true in Phases 1 and 2 by design. With high probability, the number
of turns taken in filling out each $W_i$ is $O_a(p^{a}T+p^{2(a+1)}T^{2})$. Since this is repeated $p^{-\alpha}$ times, it suffices to have
\[
T =  O_a(p^{\alpha-2(a+1)})
\]
for the number of turns to be $O(T)$. 
It remains to optimize the value of $T$ subject to 
the constraints
$T = O_a(p^{\alpha-2(a+1)})$ and
$p^{-\alpha}p^{(a+1)b+\binom{b}{2}}T^{b} = \Omega_b(1)$.
As long as $2a + 3 - b \ge 0$, this system has solutions. Solving for $\alpha$ which minimizes $T$, we find that any
\[
\alpha\le \frac{b(2a+3-b)}{2(b-1)}
\]
works, as long as the decay condition $\alpha \le 1$ was also satisfied.
\end{proof}

Lemma~\ref{lem:branch-and-fill} provides upper bounds for $f(K_m,p)$ for all $m\ge 4$, where the shape of the power of $p$ depends on the residue class of $m$ modulo $3$.

\begin{thm} \label{thm:detailed-upper-bound}
If $p\in(0,1)$, then $f(K_3,p) = O(p^{-3/2})$ and, for $m\ge4$, 
\[
f(K_m, p) = O_m(p^{-\frac{2}{3}m + c_m}),
\]
where

\[
c_m = \begin{cases}
\frac{m}{2m-3} & m\equiv0\pmod3\\
\frac{2}{3} & m\equiv1\pmod3\\
\frac{2m+8}{6m-3} & m\equiv2\pmod3.
\end{cases}
\]
\end{thm}

\begin{proof}
For $m = 3$, the bound is simple. Query $\Theta(p^{-3/2})$ pairs containing a given vertex $v_1$ and then, among the $\Theta(p^{-1/2})$ neighbors successfully found, query all pairs. For sufficiently large implied constants, the probability that we build a triangle containing $v_1$ is at least $1/2$.

When $m\ge 4$, we use Lemma~\ref{lem:branch-and-fill}, taking
\[
(a,b)=\begin{cases}
\Big(\frac{m-3}{3},\frac{2m}{3}\Big) & m\equiv0\pmod3\\
\Big(\frac{m-4}{3},\frac{2m+1}{3}\Big) & m\equiv1\pmod3\\
\Big(\frac{m-2}{3},\frac{2m-1}{3}\Big) & m\equiv2\pmod3.
\end{cases}
\]
This gives the required result.
\end{proof}

We conjecture that the bounds in Theorem~\ref{thm:detailed-upper-bound} are best possible up to the constant factor.  In the next two subsections, we prove this is the case for $m\le 5$.

\subsection{Recursive graph building} \label{sec:recursive-graph-building}

Recall that $f(H,p)$ is the number of queries
needed in the Subgraph Query Problem to build a copy of $H$ with probability at least $\frac{1}{2}$. When $H=K_m$, we can prove a lower bound on $f(H,p)$ by combining Theorem~\ref{thm:query-online-connection} with Theorem~\ref{thm:random-lower-bound}.

\begin{prop}\label{prop:f-bound-clique}
If $m\ge 3$ and $c\le\frac{1}{2}m$, then
\[
f(K_m,p) \ge \frac{1}{4} p^{-(\binom{m}{2} - c(c-1))/(m-c)}.
\]
\end{prop}
\begin{proof}
Take $N = \frac{1}{4} p^{-(\binom{m}{2} - c(c-1))/(m-c)}$, which is chosen so that
\[
p^{\binom{m}{2} - c(c-1)}(2N)^{m-c} \le \frac{1}{4}.
\]
Since $(1-p)^{\binom{n}{2}}(2N)^n\rightarrow 0$ as $n\rightarrow\infty$, there is some $n$ sufficiently large for which
\[
p^{\binom{m}{2} - c(c-1)}(2N)^{m-c}+ (1-p)^{\binom{n}{2}}(2N)^n \le \frac{1}{2}.
\]
With $d=0$, this choice of $m,n,N,p,c,d$ satisfies the conditions of Theorem~\ref{thm:random-lower-bound}, so $\tilde{r}(m,n;p) > N$. By Theorem~\ref{thm:query-online-connection}, $f(K_m,p) \ge \tilde{r}(m,n;p)$, giving the required result.
\end{proof}

We now describe a general
method for obtaining a similar lower bound on $f(H,p)$ when $H$ is not a clique. As
before, define $t(H,p,N)$ to be the maximum expected number of copies of
$H$ that can be constructed in $N$ queries. The main result of this section bounds $t(H,p,N)$ when $H$ contains a large matching. To this end, recall that a graph has a $k$-matching if it contains $k$ disjoint edges.

\begin{thm} \label{thm:t-bound}
Let $H$ be a graph containing a $k$-matching. Then there exists an absolute constant $A>1$ for which
\[
t(H,p,N)\le (Ae(H))^{e(H)} p^{e(H)-k(k-1)}(2N)^{v(H)-k},
\]
whenever $pN \ge 1$.
\end{thm}

For any edge $e\in H$,
write $H\backslash e$ for the graph formed by removing the edge
$e$ from $H$. If $U$ is a subset of the vertices of $H$, write $H\backslash U$
for the induced subgraph of $H$ on the complement of $U$. We begin by proving
the following pair of recursive bounds on $t(H,p,N)$.

\begin{lem}\label{lem:recursive-H}
If $H$ is a simple labeled graph, then
\begin{equation}
t(H,p,N)\le p\sum_{e\in E(H)}t(H\backslash e,p,N)\label{eq:recursion1}
\end{equation}
and
\begin{equation}
t(H,p,N)\le (1+o(1))pN \min_{(u,v)\in E(H)}t(H\backslash\{u,v\},p,N),\label{eq:recursion2}
\end{equation}
where the $o(1)$ term tends to $0$ as $pN \rightarrow \infty$.
\end{lem}
\begin{proof}
Suppose Builder follows an optimal strategy which achieves $t(H,p,N)$
expected copies of $H$ in $N$ turns. For each copy $H_{i}$ of $H$
that appears during the game, distinguish the edge $e_{i}$ which
is built last in $H_{i}$. For each $e\in E(H)$, let $t_{e}(H,p,N)$
be the maximum expected number of copies of $H$ that Builder can build, only counting those copies of $H$ in which $e$ is the last edge built. Then, clearly,
\[
t(H,p,N)\le\sum_{e\in E(H)}t_{e}(H,p,N).
\]
Furthermore, $t_{e}(H,p,N)\le p t(H\backslash e,p,N)$, since each
copy of $H\backslash e$ can become exactly one copy of $H$ with
success rate $p$ if $e$ is built. Inequality \eqref{eq:recursion1}
follows.

As for recursion \eqref{eq:recursion2}, note simply that the number
of copies of $H$ is bounded by the number of choices for the images
of the vertices $u,v$ which are connected by an edge times the number of copies of
$H\backslash\{u,v\}$. By the Chernoff bound, the number of choices of an edge is tightly
concentrated around $pN$, so the inequality follows.
\end{proof}

It remains to apply these inequalities recursively.

\vspace{3mm}
\noindent
{\it Proof of Theorem~\ref{thm:t-bound}. } By (\ref{eq:recursion2}), there is an absolute constant $A > 1$ for which

\begin{equation}\label{eq:recursion2-explicit}
t(H,p,N)\le ApN \min_{(u,v)\in E(H)} t(H\backslash\{u,v\},p,N)
\end{equation}
whenever $pN \ge 1$.

We proceed by induction on the number of edges in $H$. When $H$ is
an empty graph on $m$ vertices, the result is trivial with $k=0$.
Let $H$ be a labeled graph for which the induction hypothesis is
true for every graph with fewer edges than $H$. Let $e\in E(H)$ run over all edges of $H$. We break into two cases:

\vspace{2mm}
{\it Case 1.} Every $H\backslash e$ contains a $k$-matching. Then, by induction
and \eqref{eq:recursion1}, it follows that
\begin{align*}
t(H,p,N) & \le p\sum_{e\in E(H)}t(H\backslash e,p,N)\\
 & \le pe(H) (A(e(H)-1))^{e(H)-1} \cdot p^{e(H)-1-k(k-1)}(2N)^{v(H)-k}\\
 & \le (Ae(H))^{e(H)} p^{e(H)-k(k-1)}(2N)^{v(H)-k},
\end{align*}
as desired.

\vspace{2mm}
{\it Case 2.} There exists $e\in E(H)$ for which $H\backslash e$ contains no $k$-matching.
Then, let $e_{2},\ldots,e_{k}$ be $k-1$ edges which complete a $k$-matching
of $H$ containing $e$. The edges incident
to $e$ must all be incident to one of the $e_{i}$ or else $H\backslash e$
would contain a $k$-matching. Also, $e$ cannot form a $4$-cycle
with any $e_{i}$ for the same reason. From these two facts one finds
that $e$ can be incident to at most $2(k-1)$ other edges in total.
Let $H'$ be the graph obtained from $H$ by removing the two vertices
of $e$ from $H$. Applying the induction hypothesis on $H'$,
which is a graph on $v(H)-2$ vertices with at least $e(H)-(2k-1)$
edges and a $(k-1)$-matching, we find that
\[
t(H',p,N)\le (Ae(H'))^{e(H')} p^{e(H)-(2k-1)-(k-1)(k-2)}(2N)^{v(H)-2-(k-1)}.
\]
Combining this with inequality \eqref{eq:recursion2-explicit}, we have
\begin{align*}
t(H,p,N) & \le A pN\cdot t(H',p,N)\\
 & \le (Ae(H))^{e(H)} p^{e(H)-k(k-1)}(2N)^{v(H)-k},
\end{align*}
as desired.\hfill\qedsymbol

\vspace{3mm}

For our purposes, we will always assume $pN \ge 1$. Otherwise, with high probability at most a constant number of edges are built successfully in the Subgraph Query Game, so $t(H,p,N)$ will be negligibly small.

Since $t(H,p,N) < 1/2$ implies $f(H,p) > N$, Theorem~\ref{thm:t-bound} immediately implies Theorem~\ref{thm:f-bound}. Comparing this with Proposition~\ref{prop:f-bound-clique}, we note that while Theorem~\ref{thm:f-bound} gives a bound for all graphs $H$, it gives an inferior quantitative dependence on $e(H)$. While this stronger quantitative dependence in Proposition~\ref{prop:f-bound-clique} seems to be only a minor benefit, it was needed in the proof of Theorem~\ref{thm:opt-p}, which is why we retained the proof.

For large $m$, this bound only gives Corollary~\ref{cor:f-bound-asymp}, that $f(K_m, p) = \Omega_m( p^{-(2-\sqrt{2})m +O(1)})$, which is still far from the conjectured growth rate $p^{-\frac{2}{3}m+O(1)}$. However, for $m \le 5$, Theorem~\ref{thm:f-bound} can be used to pin down the asymptotic growth rate of $f(K_m, p)$, proving Theorem~\ref{thm:nle5}.

\vspace{3mm}
\noindent
{\it Proof of Theorem~\ref{thm:nle5}. }
The upper bounds for these cases are proved in Section~\ref{sec:upper-bound-strat}. Apply Theorem~\ref{thm:f-bound} by taking $k=1$ for $m=3$ and $k=2$ for $m=4, 5$ to get the desired lower bounds. \hfill \qedsymbol

\vspace{3mm}
\noindent

When $m\ge 6$, the matching argument of Theorem~\ref{thm:t-bound} does not seem sufficient for determining the exact growth rate of $f(K_m, p)$. Indeed, we will now exhibit an infinite family of graphs for which Theorem~\ref{thm:t-bound} is tight. 

For $k\ge 1$, let $H_k$ be the graph on $2k$ vertices $a_i, b_i$, $1\le i \le k$, such that $a_i \sim a_j$ for all $i\ne j$, $b_i \not\sim b_j$ for all $i \ne j$, and $a_i \sim b_j$ if and only if $i\le j$. Thus $H_k$ is a split graph consisting of a $k$-clique, a $k$-independent set, and a half graph between them. We show that Theorem~\ref{thm:t-bound} is tight for $H_k$ up to a constant factor.

Note that the construction below requires $N$ to grow like a tower of $p^{-1}$'s of height $k$. It is possible that the same lower bound is false in the regime $N \le p^{-C}$ for any $C = C(k) > 0$.

\begin{thm}\label{thm:t-bound-tight}
For every $k\ge 1$, the graph $H_k$ defined above contains a $k$-matching and, for any $p\in(0,1)$, 
\[ 
t(H_k, p, N) = \Omega_k( p^{e(H_k)-k(k-1)}N^{v(H_k)-k}),
\]
provided $N$ is sufficiently large in terms of $p$.
\end{thm}

\begin{proof}
In fact, $H_k$ has $k^2$ edges, $2k$ vertices, and contains a unique $k$-matching ${(a_i, b_i)}_{i\le k}$. It will suffice to show that for all $p\in (0,1)$ and $N$ sufficiently large in terms of $p$,
\[ 
t(H_k, p, N) = \Omega_k( p^{k}N^{k}).
\]
Builder's strategy will involve constructing a nested sequence of vertex sets $U_{1},U_{2},\ldots,U_{k}$.
The first set $U_{1}$ is just an arbitrary set of $N/k$ vertices.
In each successive $U_{i}$, assuming $|U_i| 
\ge \sqrt{N}$ we can pick $N_{i}=N/(k|U_{i}|)$ vertices
$a_{i}^{(1)},a_{i}^{(2)},\ldots,a_{i}^{(N_{i})}\in U_{i}$ and try to build
all edges from each $a_{i}^{(j)}$ to every other vertex in $U_{i}$. This step takes at most $N/k$ turns.
The set $U_{i+1}$ is then defined to be the common neighborhood of $a_{i}^{(1)},\ldots,a_{i}^{(N_{i})}$ within $U_i$. 

Repeating this process $k$ times, we use at most $N$ turns. For $N$ sufficiently large, with high probability the edge density from $a_{i}^{(1)},a_{i}^{(2)},\ldots,a_{i}^{(N_{i})}$ to the rest of $|U_i|$ is $(1+o(1))p$.
Thus, the number of copies of $H_{k}$ built in this way is bounded below by
\[
\prod_{i}(N_{i}\cdot p|U_{i}|)\ge (1+o(1))(pN)^{k}/k^{k},
\]
since we can choose $a_{i}$ out of any of the $N_{i}$ vertices $a_{i}^{(1)},\ldots,a_{i}^{(N_{i})}$
and $b_{i}$ out of any of its $(1+o(1))p|U_{i}|$ neighbors.
As long as $N$ is large enough that $|U_{k}|\ge\sqrt{N}$ with high probability, there will be enough vertices in the last set $U_k$ to perform the strategy. This argument
successfully constructs $\Theta_k(p^kN^k)$ copies of $H_k$. Taking $N$ to be a tower of $(2+2p^{-1})$'s of height $k$ is sufficient.
\end{proof}

We finish the subsection with an application of the preceding results and prove Theorem~\ref{thm:opt-p}. 
Recall that this theorem states that for $m$ fixed and $n\rightarrow\infty$, a value of $p$ for which $\tilde{r}_{\text{\normalfont rand}}(m,n) \le 3\tilde{r}(m,n;p)$ satisfies $p = \Theta(\frac{m}{n}\log{\frac{n}{m}})$.

\vspace{3mm}
\noindent
{\it Proof of Theorem \ref{thm:opt-p}.}
By Theorem~\ref{thm:detailed-upper-bound}, 
\begin{equation*}\label{eq:f-known}
f(K_m,p) = O_m(p^{-2m/3}),
\end{equation*}
and in fact it can be checked from the proof that the explicit dependence on $m$ is polynomial. Moreover, using Proposition~\ref{prop:f-bound-clique} with $c=0$, we have that
\[
f(K_m,p) \ge \frac{1}{4} p^{-\frac{m-1}{2}} \ge \frac{1}{4}p^{-m/3},
\]
since $\frac{m-1}{2}\ge \frac{m}{3}$ for $m\ge 3$. Putting all this together, there exists an absolute constant $A>0$ for which
\begin{equation}\label{eq:f-known-explicit}
\frac{1}{4} p^{-m/3} \le f(K_m,p) \le m^A p^{-2m/3}
\end{equation}
for all $m\ge 3$, $p\in (0,1)$. 

By Theorem~\ref{thm:query-online-connection}, we have
\[
\tilde{r}(m,n;p) \le \min\{f(K_m, p), f(K_n, 1-p)\} \le 3\tilde{r}(m,n;p).
\]
Pick some $p_0 \in (0,1)$ which maximizes the function $\min\{f(K_m, p), f(K_n, 1-p)\}$. Such a $p_0$ exists because $f(K_m,p)$ is nonincreasing in $p$, $f(K_n,1-p)$ is nondecreasing, and both are integer-valued. Then, $\tilde{r}_{\text{\normalfont rand}}(m,n)\le 3\cdot \tilde{r}(m,n;p_0)$. It remains to check that we could have chosen $p_0 = \Theta(\frac{m}{n}\log{\frac{n}{m}})$. By (\ref{eq:f-known-explicit}) and the fact that the bounds are continuous, we have
\[
\frac{1}{4}p_0^{-\frac{1}{3}m} \le n^A (1-p_0)^{-\frac{2}{3}n}
\]
and
\[
\frac{1}{4}(1-p_0)^{-\frac{1}{3}n} \le m^A p_0^{-\frac{2}{3}m}.
\]

Since $m\ge 3$ is fixed and $n\rightarrow \infty$, the first inequality implies $p_0\rightarrow 0$. In particular, $\log(1-p_0)=-p_0 +O(p_0^2)$. Taking the logarithm of both sides in the inequalities above, we have
\[
-\frac{1}{3} m \log p_0 -\log 4 \le A \log n + \frac{2}{3} n (p_0 +O(p_0 ^2))
\]
and
\[
\frac{1}{3}n(p_0 +O(p_0 ^2)) -\log 4 \le A \log m - \frac{2}{3}m\log p_0.
\]

Taking $n\rightarrow \infty$ and dividing through by $mp_0$, these inequalities combine to show
\[
\frac{\log (1/p_0)}{p_0} = \Theta\left(\frac{n}{m}\right)
\]
and it follows that $p_0 = \Theta(\frac{m}{n}\log{\frac{n}{m}})$, as desired. \hfill\qed

\vspace{3mm}

\subsection{The value of $t(K_m,p,N)$ for large $N$}\label{sec:t-large-N}

In this section, we investigate the behavior of the function $t(K_m,p,N)$ as $N\rightarrow \infty$. We find that when $N$ is very large, the essentially optimal strategy for building as many copies of $K_m$ as possible is to fill in the edges of a clique on $\sqrt{2N}$ vertices. This is in stark contrast with the rather delicate procedure described in Section~\ref{sec:upper-bound-strat} to build a single copy of $K_m$.

\subsubsection{Chernoff bounds and subjumbledness}

We will need a standard lemma (see, for example,~\cite[Theorem 2.1]{KrSu}) saying that with high probability all moderately large induced subgraphs of a random graph $G(N,p)$ have the expected number of edges. Recall that if $U\subset V(G)$ is a vertex subset of $G$, we write $G[U]$ for the induced subgraph on $U$.

\begin{lem}
\label{lem:chernoff}If $G=G(N,p)$ and $\varepsilon>0$, then, with high probability,
\[
e(G[U])=(1\pm\varepsilon)p\binom{|U|}{2}
\]
for all $|U|=\Omega_{\varepsilon}(p^{-1}\log N)$. 
\end{lem}

In the literature (see~\cite{KrSu} and its references), this pseudorandomness property is usually called jumbledness. We also use this term, though in a slightly different way to how it is usually used.

\begin{defn}
A graph $G$ is {\it $(p,M,\varepsilon)$-jumbled} if, for every $U\subseteq V(G)$
with $|U|\ge M$,
\[
e(G[U])=(1\pm\varepsilon)p\binom{|U|}{2}.
\]
A graph $G$ is {\it $(p,M,\varepsilon)$-subjumbled} if it is a subgraph of some $(p,M,\varepsilon)$-jumbled graph.
\end{defn}

In what follows, we will show that subjumbled graphs
cannot have too many cliques. For the graph-building problem, the
heuristic is that it's not possible to build more copies of $H$ in
a known jumbled graph $G$ with $pN$ queries than it is 
with $N$ queries in $G(N,p)$.

\subsubsection{Degeneracy}

Define a graph to be $d$-degenerate if there exists an ordering of
the vertices $v_{1},\ldots,v_{n}$ such that $|N(v_{i})\cap\{v_{1},\ldots,v_{i-1}\}|\le d$ for all $i$. The following simple lemma is well known.

\begin{lem}
\label{lem:degeneracy}Every graph with $E$ edges is $\sqrt{2E}$-degenerate.
\end{lem}

\begin{proof}
We exhibit the degenerate ordering by picking the vertices backwards
from $v_{n}$ to $v_{1}$. At each step, pick $v_{i}$ to be the minimal
degree vertex in the current graph and delete it. Note that $d(v_{i})\le i-1$
because there are only $i$ points left and also $d(v_{i})\le\frac{2E}{i}$
because the sum of the degrees is at most $2E$ and $v_{i}$ has minimal
degree. It follows that at every step $d(v_{i})\le\min(i,\frac{2E}{i})\le\sqrt{2E}$, as desired.
\end{proof}

This is not quite sufficient for our purposes, but it gives the main idea. What we really need is a better understanding of degeneracy in jumbled graphs.
In what follows, given a candidate ordering $v_{1},\ldots,v_{n}$ of the vertices,
we write $N^{-}(v_{i})=N(v_{i})\cap\{v_{1},\ldots,v_{i-1}\}$ and $d^{-}(v_{i})=|N^{-}(v_{i})|$.

\begin{lem} \label{lem:main-degeneracy}
Any $(p,M,\varepsilon)$-subjumbled graph on $N$ edges is $\max((1+\varepsilon)M\sqrt{p},(1+\varepsilon)\sqrt{2pN})$-degenerate.
\end{lem}

\begin{proof}
Let $H$ be a graph on $N$ edges that is a subgraph of some $(p,M,\varepsilon)$-jumbled graph $G$. We again pick vertices
of the graph $H$ in order of increasing degree among the remaining
vertices. Let the resulting order be $v_{1},\ldots,v_{n}$ and write
$U_{i}=\{v_{1},\ldots,v_{i}\}$. The construction guarantees that $v_{i}$ is
of minimal degree in $G[U_{i}]$. 

If $i\le M$, then the subgraph $H[v_{1},\ldots,v_{i}]$ has at most
as many edges as $H[v_{1},\ldots,v_{M}]$, which has at most $(1+\varepsilon)pM^{2}/2$
edges. Thus,
\begin{align*}
d^{-}(v_{i}) & \le \min\Big(i,\frac{(1+\varepsilon)pM^{2}}{i}\Big)\\
 & \le (1+\varepsilon)M\sqrt{p}.
\end{align*}
Otherwise, if $i>M$, the induced subgraph $H[v_{1},\ldots,v_{i}]$
has at most $(1+\varepsilon)pi^{2}/2$ edges and it clearly cannot have more
than $e(H)=N$ edges. Because $v_{i}$ is of minimal degree
in this induced subgraph,
\begin{align*}
d^{-}(v_{i}) & \le \frac{2}{i}\min\Big(\frac{(1+\varepsilon)pi^{2}}{2},N\Big)\\
 & = \min((1+\varepsilon)pi,2Ni^{-1})\\
 & \le (1+\varepsilon)\sqrt{2pN}
\end{align*}
and so every vertex has $d^{-}(v_{i})\le\max((1+\varepsilon)M\sqrt{p},(1+\varepsilon)\sqrt{2pN})$,
as desired.
\end{proof}

\subsubsection{Counting cliques}

We are ready to prove the following lemma. Recall the standard notation that $t(K,H)$ is the number of labeled graph homomorphisms from $K$ to $H$. Up to a lower order term, this is the same as counting labeled copies of $K$ in $H$. In fact, the equality is exact in the case we care about, where $K$ is a clique and $H$ is a simple graph without self-loops.

\begin{lem} \label{lem:t-bound}
For all $p\in(0,1)$, $m,M\ge2$ and $0<\varepsilon<1$, if $H$ is
a $(p,M,\varepsilon)$-subjumbled graph with $N$ edges, then
\begin{equation*}
t(K_{m},H)\le(1+O_{m}(\varepsilon+p^{1/2}N^{-1/2}))p^{\binom{m}{2}}(2p^{-1}N)^{\frac{m}{2}}+O_{m}\Big(\sum_{k=2}^{m-1}p^{\frac{m+k(k-3)}{2}}\cdot M^{m-k}\cdot N^{\frac{k}{2}}\Big).\label{eq:lem:t-bound}
\end{equation*}
\end{lem}

\begin{proof}
Take a degenerate ordering $v_{1},\ldots,v_{n}$ of $H$ such that
$v_{i}$ is of minimum degree in $H[v_{1},\ldots,v_{i}]$. By Lemma~\ref{lem:main-degeneracy}, 
\[
d^{-}(v_{i})\le\max((1+\varepsilon)M\sqrt{p},(1+\varepsilon)\sqrt{2pN}),
\]
where the second term dominates as soon as $N\ge M^{2}/2$.
Conditioning on whether or not $v_{n}$ is in the copy of $K_{m+1}$
we are counting, we see that
\[
t(K_{m+1},H)-t(K_{m+1},H\backslash v_{n})=(m+1)t(K_{m},H[N^{-}(v_{n})]).
\]
In particular, writing $t(m,N)=\max_{e(H)=N}^{*}t(K_{m},H)$, where
the maximum is taken over all graphs $H$ with $N$ edges that are subgraphs of
some $(p,M,\varepsilon)$-jumbled graph, we find that
\begin{equation}
t(m+1,N)\le\max_{d\le U(N)}\Big[t(m+1,N-d)+(m+1)t(m,e^{+}(d))\Big],\label{eq:t-system}
\end{equation}
where $U(N)=\max((1+\varepsilon)M\sqrt{p},(1+\varepsilon)\sqrt{2pN})$
and $e^{+}(d)$ is any upper bound on the number of edges in a graph on $d$
vertices that is a subgraph of a $(p,M,\varepsilon)$-jumbled graph.
The function $e^{+}$ we take is
\[
e^+(d) = \begin{cases}
\frac{d^{2}}{2} & d<M\sqrt{p}\\
(1+\varepsilon)\frac{pM^{2}}{2} & M\sqrt{p}\le d<M\\
(1+\varepsilon)\frac{pd^{2}}{2} & d\ge M.
\end{cases}
\]
To see that $e^+(d)$ is indeed an upper bound on the number of edges in a graph on $d$ vertices that is a subgraph of a $(p,M,\varepsilon)$-jumpled graph, we use the trivial bound when $d$ is small, extend to a
size $M$ set to use jumbledness when $d$ is somewhat close to $M$, and use
jumbledness directly for $d$ larger than $M$.

We are left to bound $t$ using the system of inequalities (\ref{eq:t-system}).
Write
\[
t^*(m,N) =  p^{\binom{m}{2}}(2p^{-1}N)^{\frac{m}{2}}
\]
for the approximate optimum value of $t(m,N)$. We induct on $m$. The base case is $t(2,N)=2N$. Assume, by induction, that for some $m\ge 2$,
\[
t(m,N)\le(1+O_{m}(\varepsilon+p^{1/2}N^{-1/2}))t^*(m,N)+O_{m}\Big(\sum_{k=2}^{m-1}p^{\frac{m+k(k-3)}{2}}\cdot M^{m-k}\cdot N^{\frac{k}{2}}\Big).
\]

We would like to show that the same inequality holds for $m+1$. Iterating (\ref{eq:t-system}), there exists a sequence $(d_i)_{i\ge 1}$ of positive integers summing to $N$ for which
\[
d_i \le U\Big(N-\sum_{1\le j < i} d_j\Big)
\]
and
\[
t(m+1,N) \le (m+1)\sum_{i\ge 1} t(m,e^+(d_i)),
\]
which implies, by the induction hypothesis, that
\begin{align}
t(m+1,N) & \le (1+O_m(\varepsilon + p^{1/2}N^{-1/2}))(m+1)\sum_{i\ge 1}t^*(m,e^+(d_i)) \nonumber \\
& + O_{m+1}\Big(\sum_{k=2}^{m-1}p^\frac{m+k(k-3)}{2} \cdot M^{m-k} \cdot \Big[\sum_{i\ge 1}e^+(d_i)^{\frac{k}{2}}\Big]\Big).\label{eq:t-recursive}
\end{align}

Since $e^+(d)$ is constant on the range $M\sqrt{p}\le d < M$, the optimal choice of $d_i$ will never have any points in this range. The main term of (\ref{eq:t-recursive}) can thus be separated into the sum over $d_i < M\sqrt{p}$ and the sum over $d_i \ge M$:
\begin{equation} \label{eq:ranges}
\sum_{i\ge 1}t^*(m,e^+(d_i)) \le \sum_{d_i \ge M} t^*\Big(m, (1+\varepsilon)pd_i^2/2\Big) + \sum_{d_i < M\sqrt{p}} t^*\Big(m, (1+\varepsilon)d_i^2/2\Big).
\end{equation}
Note that $m\ge 2$, so $t^*(m,N)$ is a convex nondecreasing function in $N$. Also, the function $e^{+}(d)$ is nondecreasing and convex in $d$ except for the jump discontinuity at $d=M\sqrt{p}$. Therefore, in each of the ranges above, $t^*(m, \cdot)$ and $e^+(\cdot)$ are both convex nondecreasing functions.

To bound the first sum in (\ref{eq:ranges}), we pass to an integral. Write 
\[
N_i = N - \sum_{j\le i} d_i.
\]
Then
\begin{align*}
\sum_{d_i \geq M} t^*(m,e^+(d_i)) & \le \sum_{i\ge 1} t^*\Big(m, (1+\varepsilon)pd_i^2/2\Big) \\
& = \sum_{i\ge 1} \int_{N_i}^{N_{i-1}}\frac{t^*(m, (1+\varepsilon)pd_i^2/2)}{d_i} dx.
\end{align*}
Because $t^*(m,(1+\varepsilon)pd^2/2)/d$ is an increasing function of $d$ and $d_i \le U(N_{i-1}) = (1+\varepsilon)\sqrt{2pN_{i-1}}$, we have
\begin{align*}
\sum_{i\ge 1} \int_{N_i}^{N_{i-1}}\frac{t^*(m, (1+\varepsilon)pd_i^2/2)}{d_i} dx 
& \le \sum_{i\ge 1} \int_{N_i}^{N_{i-1}}\frac{t^*(m, (1+\varepsilon)p((1+\varepsilon)\sqrt{2pN_{i-1}})^2/2)}{(1+\varepsilon)\sqrt{2pN_{i-1}}} dx \\
& \le  \sum_{i\ge 1} \int_{N_i+(1+\varepsilon)\sqrt{2pN}}^{N_{i-1}+(1+\varepsilon)\sqrt{2pN}}\frac{t^*(m, (1+\varepsilon)p((1+\varepsilon)\sqrt{2px})^2/2)}{(1+\varepsilon)\sqrt{2px}} dx \\
& \le \int_{0}^{N+(1+\varepsilon)\sqrt{2pN}}\frac{t^*(m,(1+\varepsilon)^3p^{2}x)}{\sqrt{2px}}dx.
\end{align*}
We had to shift integrals in the second step to guarantee that every value of $x$ in the range of integration is at least $N_{i-1}$. 

Next, $t^*(m,N)$ is a polynomial in $N$, so we can absorb the $(1+\varepsilon)$ into the error term. Similarly, we can pull out an error term of $(1+(1+\varepsilon)\sqrt{2p/N})$ from the bounds of the integral to simplify. Reorganizing various error terms, we get
\[
\int_{0}^{N+(1+\varepsilon)\sqrt{2pN}}\frac{t^*(m,(1+\varepsilon)p^{2}x)}{\sqrt{2px}}dx
\le (1+O_{m+1}(\varepsilon + p^{1/2}N^{-1/2}))\int_{0}^{N}\frac{t^*(m,p^{2}x)}{\sqrt{2px}}dx.
\]
Finally, explicitly evaluating the integral, we have
\begin{align*}
\int_{0}^{N}\frac{t^*(m,p^{2}x)}{\sqrt{2px}}dx & = \int_{0}^{N}p^{\binom{m}{2}}(2p^{-1}p^{2}x)^{\frac{m}{2}}\frac{dx}{\sqrt{2px}}\\
 & = 2^{\frac{m-1}{2}}p^{\frac{m^{2}-1}{2}}\int_{0}^{N}x^{\frac{m-1}{2}}dx\\
 & = \frac{1}{m+1}2^{\frac{m+1}{2}}p^{\frac{m^{2}-1}{2}}x^{\frac{m+1}{2}}\Big|_{0}^{N}\\
 & = \frac{1}{m+1}p^{\binom{m+1}{2}}(2p^{-1}N)^{\frac{m+1}{2}} \\
 & = \frac{1}{m+1}t^*(m+1, N).
\end{align*}

Estimating the second sum in (\ref{eq:ranges}) trivially, we get

\begin{equation*}
\sum_{i\ge 1}t^*(m,e^+(d_i)) \le (1+O_{m+1}(\varepsilon+p^{1/2}N^{-1/2}))t^*(m+1,N)+O_{m+1}\Big(\frac{N}{M\sqrt{p}}t^*\Big(m,\frac{pM^{2}}{2}\Big)\Big).
\end{equation*}

To check the error terms in (\ref{eq:t-recursive}) match up
is similar: break up each sum into the sums over $d_i \ge M$ and $d_i < M\sqrt{p}$. The first sum is estimated by an integral and the second trivially. The result is
\[
O_{m+1}\Big(\sum_{k=2}^{m-1}p^\frac{m+k(k-3)}{2} \cdot M^{m-k} \cdot \Big[\sum_{i\ge 1}e^+(d_i)^{\frac{k}{2}}\Big]\Big) \le O_{m+1}\Big(\sum_{k=2}^{m}p^{\frac{m+1+k(k-3)}{2}}\cdot M^{m+1-k}N^{\frac{k}{2}}\Big),
\]
which is the right error term for $t(m+1,N)$, completing the induction.
\end{proof}

In particular, and this is essential, the implicit constants in this lemma do not
depend on $M$. As an immediate corollary, we now prove Theorem~\ref{thm:t-large-N}. Note that the $N$ above is the number of edges in $H$, which will correspond to $(1+o(1))pN$ below if $N$ is the number of queries made in the Subgraph Query Game.

\vspace{3mm}
\noindent
{\it Proof of Theorem~\ref{thm:t-large-N}.}
Applying the Chernoff bound from Lemma~\ref{lem:chernoff}, we see
that for any $\varepsilon>0$ we can take some $M=Cp^{-1}\log N$ so that the random graph $G(2N,p)$ is $(p,M,\varepsilon)$-jumbled with high probability. Also with high probability, the number of edges built in $N$ queries is $(1+o(1))pN$. It is easy to check that the exponentially small probabilities with which either of these are false have negligible impact on the value of $t(K_m,p,N)$. The subgraph $H$ built by Builder
must therefore satisfy the hypotheses of Lemma~\ref{lem:t-bound}
with $(1+o(1))pN$ edges.

The main term dominates the error terms for $N$ sufficiently large, giving the expected answer which is just $p^{\binom{m}{2}}(2N)^{\frac{m}{2}}$,
the number of $m$-cliques in $G(\sqrt{2N},p)$. This happens once the main term outgrows the largest error term, the term with $k=m-1$. This happens at $N = \Omega(p^{-(2m-3)}M^2)$, so it suffices to have $N \ge \omega(p^{-(2m-1)}\log^2{(p^{-1})})$. This proves the upper bound in Theorem~\ref{thm:t-large-N}. Of course, the lower bound is proved by the strategy of building all edges among $\sqrt{2N}$ vertices. \hfill \qedsymbol

\section{Concluding remarks} \label{sec:concluding-remarks}

It is an interesting problem to close the gap in the bounds for the online Ramsey number $\tilde{r}(m,n)$. In particular, we know that there are positive constants $c,c'$  for which $cn^3/(\log n)^2 \leq \tilde{r}(3,n) \leq c' n^3$ and it seems plausible that these bounds could be brought closer together. Indeed, we conjecture that the lower bound can be improved to $cn^3/\log n$ by considering the following Painter strategy motivated by the triangle-free process~\cite{B09}. Painter applies the triangle-free process to obtain an auxiliary triangle-free graph $G$ on vertex set $\{1, 2, \dots, r\}$ with $r=c_0 n^2/\log n$. Painter does not reveal this auxiliary graph. As before, we label vertices that reach degree $n/4$ with $1,\ldots,r$ as they arrive at degree $n/4$. When Builder adds an edge between two vertices in which both vertices have degree at least $n/4$, then these vertices have labels, say $i$ and $j$, and Painter paints the edge with the color of the edge $ij$ in $G$. Otherwise, they color the edge blue. This coloring clearly contains no red triangles, but it remains to show that it contains no blue $K_n$.

In studying the online Ramsey number, we were usually led by the idea that Builder's optimal strategy is to fill out an extremely sparse graph on the vertex set they touch. However, if Builder is restricted to play on a small vertex set, this intuition seems to go awry. If we define $\tilde{r}(m,n;N)$ in the same manner as the online Ramsey number but with the additional restriction that only $N$ vertices are allowed, then we conjecture that the function $\tilde{r}(m,n;N)$ increases substantially as $N$ decreases from $2 \tilde{r}(m,n)$, the maximum number of vertices in a graph with $\tilde{r}(m,n)$ edges, down to its minimal meaningful value $r(m,n)$.

The order of growth of $f(K_m, p)$ is still open for $m\ge 6$. In particular, Theorems~\ref{thm:f-bound} and~\ref{thm:detailed-upper-bound} show that $f(K_6, p) = \Omega(p^{-13/4})$ and $f(K_6,p) = O(p^{-10/3})$ and we conjecture that the upper bound is correct. This belief is rooted in our conviction that the upper bound for $t(H,p,N)$ given by Theorem~\ref{thm:t-bound}, upon which Theorem~\ref{thm:f-bound} relies, is not tight when $N$ is on the order of $f(H,p)$. Because of the examples in Theorem~\ref{thm:t-bound-tight}, these upper bounds can be tight when $N$ is very large, so further progress on this problem would need to be more sensitive to the size of $N$. It is plausible that any advance on this question and its generalizations could also impinge on our estimates for online Ramsey numbers.

\medskip

\noindent {\bf Acknowledgements.} We are extremely grateful to Joel Spencer for pointing out a serious flaw in our previous proof of Theorem~\ref{thm:alterations} which had been based on a generalization of the Lov\'asz Local Lemma~\cite{ErSp}. In the current version, we have a correct proof using a different approach. We would also like to thank the anonymous referee for some helpful remarks and Benny Sudakov for bringing the paper of Krivelevich~\cite{Kr1} to our attention.

\end{document}